\documentclass[onefignum,onetabnum]{siamart220329}

\usepackage{diagbox}
\usepackage{multirow}
\usepackage{amsfonts,amssymb} 
\usepackage{color} 
\usepackage{graphicx} 
\usepackage{epsfig,mathrsfs} 
\usepackage[bf,SL,BF]{subfigure} 
\usepackage{fancyhdr} 
\usepackage{CJK} 
\usepackage{caption} 
\usepackage{wrapfig} 
\usepackage{cases} 
\usepackage{bm}
\usepackage{algorithm}
\usepackage{algorithmic}

\ifpdf
\DeclareGraphicsExtensions{.eps,.pdf,.png,.jpg}
\else
\DeclareGraphicsExtensions{.eps}
\fi

\newsiamremark{remark}{Remark}
\newsiamremark{condition}{Condition}
\newsiamremark{example}{Example}
\newsiamremark{hypothesis}{Hypothesis}
\crefname{hypothesis}{Hypothesis}{Hypotheses}
\newsiamthm{claim}{Claim}

\headers{ID$m$-BDF$k$ for anomalous diffusion with hyper-singular term}{J. Shi, M. Chen, and J. Cao}

\title{High-order BDF convolution quadrature for fractional evolution equations with hyper-singular source term
	\thanks{Submitted to the editors DATE.
		\funding{This work was supported by the Science Fund for Distinguished Young Scholars of Gansu Province under Grant No. 23JRRA1020 and the Fundamental Research Funds for the Central Universities under Grant No.
			lzujbky-2023-06.}}}

\author{Jiankang Shi\thanks{School of Mathematics and Statistics, Gansu Key Laboratory of Applied Mathematics and Complex Systems, Lanzhou University, Lanzhou 730000, P.R. China (\email{shijk20@lzu.edu.cn}, \email{chenmh@lzu.edu.cn}).}
	\and Minghua Chen \footnotemark[2]
	\and Jianxiong Cao\thanks{School of Sciences, Lanzhou University of Technology, Lanzhou 730000, P.R. China (\email{caojianxiong2007@126.com}).}}

\usepackage{amsopn}

\ifpdf
\hypersetup{
	pdftitle={ID$m$-BDF$k$ for anomalous diffusion with hyper-singular term},
	pdfauthor={J. Shi, M. Chen, and J. Cao}
}
\fi

\begin{document}
	
	\maketitle
	
	\begin{abstract}
		Anomalous diffusion in the presence or absence of an external force field is often modelled in terms of the fractional evolution equations, which can involve the hyper-singular source term.
		For this case, conventional time stepping methods may exhibit a severe order reduction. Although   a second-order  numerical algorithm is provided for the subdiffusion model with a simple hyper-singular source term $t^{\mu}$, $-2<\mu<-1$ in [arXiv:2207.08447],   the convergence analysis remain to be proved.
		To fill in these gaps, we present a simple and robust smoothing method for the hyper-singular source term, where the Hadamard finite-part integral is introduced. This method is based on the smoothing/ID$m$-BDF$k$ method proposed by the authors [Shi and Chen, SIAM J. Numer. Anal., to appear] for subdiffusion equation with a weakly singular source term.
		We prove that the $k$th-order convergence rate can be restored for the diffusion-wave case $\gamma \in (1,2)$ and sketch the proof for the subdiffusion case $\gamma \in (0,1)$, even if the source term is hyper-singular and the initial data is not compatible.
		Numerical experiments are provided to confirm the theoretical results.
	\end{abstract}
	
	
	
	\section{Introduction}\label{Se:intro}
	We are interested in the convolution quadrature (CQ) generated by the $k$-step backward differentiation formulas (BDF$k$) for solving the fractional evolution equation with the hyper-singular source term, whose prototype equation is, for $1<\gamma<2$
	\begin{equation} \label{fee}
		\partial^{\gamma}_t (u(t)- \upsilon - tb) - A u(t) = g(t):= t^{\mu}\circ f(t), \quad -2<\mu<-1
	\end{equation}
	with the initial condition $u(0)=\upsilon$, $u'(0)=b$. The operator $A$ is the Laplacian $\Delta$ on a bounded, convex domain $\Omega \subset \mathbb{R}^d$ ($d=1,2,3$), with homogenous Dirichlet boundary condition.
	Here $\mathcal{D}(A) = H^{1}_{0}(\Omega) \cap H^{2}(\Omega)$, and $H^{1}_{0}(\Omega)$, $H^{2}(\Omega)$ denote the standard Sobolev spaces \cite{ThomeeGalerkin2006}.
	The symbol $\circ$ either the product or the convolution, and the Riemann-Liouville fractional derivative of order $\gamma$, is defined by \cite[p.\,62]{PodlubnyFractional1999}
	\begin{equation}\label{RL1.2}
		\partial^{\gamma}_t u(t)= \frac{1}{\Gamma(2-\gamma)}\frac{d^{2}}{dt^{2}} \int^{t}_{0} {(t-\tau)^{1-\gamma}u(\tau)} d\tau, \quad 1<\gamma<2.
	\end{equation}
	It makes sense to allow $\partial^{\gamma}_t u(t)$ to be hyper-singular at $t=0$ if $u$ is absolutely continuous, e.g.,
	\begin{equation*}
		\partial^{\gamma}_t 1=\frac{1}{\Gamma(1-\gamma)}t^{-\gamma}\rightarrow \infty \quad\rm{as}\quad t \rightarrow 0, \quad 1<\gamma<2.
	\end{equation*}
	This leads to the fractional evolution equations involving the hyper-singular source term, see \cite[Eq.(21)]{HeymansPhysical2006}, \cite[Eq.(4.2.57)]{KilbasTheory2006}, \cite[Eq.(4)]{OzdemirFractional2008} and \cite[Eq.(10)]{Comptegeneralized1997}.
	
	Problems of the model \eqref{fee} arise in many areas of the applied sciences, such as the transport of chemical contaminants, through water around rocks, and the dynamics of viscoelastic materials \cite{HeymansPhysical2006,KilbasTheory2006,MainardiFractional2022,PovstenkoFractional2015}.
	
	It is well-known that the analytical solution of \eqref{fee} has an initial layer at $t=0$ and $\partial u(x,t)/\partial t$ blows up as $t\rightarrow 0$ even if all the data of \eqref{fee} is sufficiently smooth \cite{JinCorrection2017,SakamotoInitial2011,StynesError2017}, which may deteriorate the convergence rate of the numerical methods.
	Nowadays, there are two popular time-stepping methods to recover the high-order convergence rate for the fractional evolution equation \eqref{fee} under the mild regularity of the source function. The first way is the variable time-stepping schemes (e.g., geometric meshes, graded meshes), which is powerful in capturing the singularities of the solution at $t=0$, see \cite{ChenTwo2021,KoptevaError2021,LiaoSharp2018,MustaphaTimestepping2015,Mustaphaapproximation2020,StynesError2017} for the subdiffusion case $\gamma \in (0,1)$ and \cite{MustaphaSuperconvergence2013,MustaphaWellposedness2014} for the diffusion-wave case $\gamma \in (1,2)$.
	The second way is convolution quadrature generated by BDF$k$ or Lagrange interpolation with degree $k$, see \cite{JinCorrection2017,ShiCorrection2020,ShiCorrection2022,Yananalysis2018} for $\gamma \in (0,1)$ and \cite{CuestaConvolution2006,JinTwo2016,JinCorrection2017,LubichNonsmooth1996} for $\gamma \in (1,2)$.
	It is important to note that, for the low regularity source term, the correction of high-order BDF$k$ schemes \cite{CuestaConvolution2006,JinTwo2016,JinCorrection2017,LubichNonsmooth1996,ShiCorrection2020,Yananalysis2018} may suffer form a severe order reduction. For example, it is reduces to the order $\mathcal{O}(\tau^{1+\mu})$ for the source term $g(t)=t^\mu$, $0<\mu<1$, see Lemma 3.2 in \cite{WangHighorder2020}.
	
	The study on the weakly singular source function for the subdiffusion model is scarce.
	In \cite{ZhouTwo2022}, a second-order method is presented for the singular source function $g(t)=t^{\mu}$, $\mu>-1$ by performing an integral-differential operator on both sides of the subdiffusion equation.
	An optimal error estimate of a high-order BDF convolution quadrature is provided in \cite{ShiHighorder2023} with $g(t):=t^{\mu}\circ f(t)$, $\mu>-1$, where the singular source term is regularized by $m$-fold integral-differential operators (ID$m$) and the equation is discretized by the BDF$k$, called smoothing method or ID$m$-BDF$k$ method.
	
	To the best of our knowledge, we are unaware of any other published work on the hyper-singular source function for the fractional evolution equations \eqref{fee} including the subdiffusion case.
	Although the numerical algorithm is provided for the subdiffusion case with $g(t)=t^{\mu}$, $-2<\mu<-1$, where the convergence analysis remain to be proved \cite{ChenModified2022}.
	
	To fill in this gap, we present a simple and robust smoothing method for the hyper-singular source term, where the Hadamard finite-part integral \cite{Diethelmanalysis2010} is introduced.
	This method is based on ID$m$-BDF$k$ method proposed by the authors \cite{ShiHighorder2023} for subdiffusion equation with a weakly singular source term.
	We prove that the $k$th-order convergence rate can be restored for the diffusion-wave case $\gamma \in (1,2)$ and sketch the proof for the subdiffusion case $\gamma \in (0,1)$, even if the source term is hyper-singular and the initial data is not compatible.
	Numerical experiments are provided to confirm the theoretical results.

	\section{Preliminaries: Numerical scheme and solution representation}\label{Se:corre}
	Let $V(t)=u(t)- \upsilon -tb$ with $V(0)=0$. Then \eqref{fee} can be recast (with $1<\gamma<2$)
	\begin{equation}\label{rfee}
		\partial^{\gamma}_{t} V(t) - A V(t)= A\upsilon + tAb + g(t), \quad 0<t\leq T.
	\end{equation}
	
	As is well known, the linear operator $A$ satisfies the resolvent estimate \cite{ThomeeGalerkin2006}
	\begin{equation*}\label{resolvent estimate}
		\left\| (z-A)^{-1} \right\| \leq c |z|^{-1} \quad \forall z\in \Sigma_{\theta}
	\end{equation*}
	for all $\theta\in (\pi/2, \pi)$. Here $\Sigma_{\theta}:=\{ z\in \mathbb{C}\backslash \{0\}:|\arg z| < \theta \}$ is a sector of the complex plane $\mathbb{C}$.
	Choose the angle $\theta$ such that $\pi/2 < \theta < \min{(\pi, \pi/\gamma)}$, and it holds
	\begin{equation}\label{fractional resolvent estimate}
		\big\| \left(z^{\gamma}-A\right)^{-1} \big\| \leq c |z|^{-\gamma} \quad \forall z\in \Sigma_{\theta}.
	\end{equation}
	Here and below $\left\|\cdot \right\|$ and $\left\|\cdot \right\|_{L^2(\Omega)}$, respectively, denote the operator norm \cite[p.\,91]{ThomeeGalerkin2006} and usual norm \cite[p.\,2]{ThomeeGalerkin2006} in the space $L^2(\Omega)$.
	
	\subsection{ID$m$-BDF$k$ method}
	Let us first introduce the Hadamard’s finite-part integral for $t^{-\beta}$
	\begin{equation}\label{Hadmint}
		t^{-\beta} \circledast 1:= \oint_{0}^{t} s^{-\beta} ds = \frac{t^{1-\beta}}{1-\beta} \quad {\rm with} \quad \beta >1.
	\end{equation}
	Estimates of finite-part integral \eqref{Hadmint} obtained by using the regularization formulas, where Hadamard suggested simply to ignore the unbounded contribution, see \cite[p.\,233]{Diethelmanalysis2010}.
	
	Performing the $m$-fold integral operator for $g(t)$, it yields \cite[p.\,193]{PodlubnyFractional1999}
	\begin{equation}\label{smoothing1}
		G(t)=J^{m}g(t)=\frac{t^{m-1}}{\Gamma(m)}\circledast g(t) :=\frac{1}{\Gamma(m)}\oint_0^t(t-\tau)^{m-1}g(\tau)d\tau.
	\end{equation}
	If the integrals in the right-hand side of \eqref{smoothing1} exist in the usual sense, then the value of the right-hand side gives the finite value of the integral, which reduces to the standard integral.
	Here the symbols $\oint$ (or $\int$) and $\circledast$ (or $\ast$) denote the Hadamard finite-part integral (or the standard integral) and Hadamard finite-part convolution (or the standard convolution), respectively.

	Then the fractional diffusion-wave equation \eqref{rfee} can be rewritten as (with $2\leq m \leq k+1$, $k\leq 6$)
	\begin{equation}\label{rrfee}
		\partial^{\gamma}_{t} V(t) - AV(t)=\partial^{m}_{t} \left(\frac{t^{m}}{\Gamma(m+1)} A\upsilon + \frac{t^{m+1}}{\Gamma(m+2)} Ab + G(t)\right), \quad 0<t\leq T.
	\end{equation}
	It is important to note that $G(0)=J^m g(t)|_{t=0}=0$, $m\geq2$, e.g., $J^m$ may map the hyper-singular point of $g$ to a zero point of $G$.
	
	Let $t_{n}=n \tau$, $n=0,1,\ldots,N$, be a uniform partition of the time interval $[0,T]$ with the step size $\tau=\frac{T}{N}$.
	Let $u^{n}$ be the approximation of $u(t_{n})$ and $g^{n}=g(t_{n})$ at $t=t_{n}$.
	The Riemann-Liouville fractional derivative $\partial^{\gamma}_{t}V(t_n)$ in \eqref{RL1.2} can be approximated by the convolution quadrature \cite{LubichDiscretized1986}
	\begin{equation}\label{2.1}
		\partial^{\gamma}_{\tau, k}V^{n}:=\frac{1}{\tau^{\gamma}}\sum^{n}_{j=0} \omega_{j}^{(\gamma, k)}V^{n-j}, \quad1\leq k\leq 6
	\end{equation}
	Here the convolution quadrature weights $\omega_{j}^{(\gamma, k)}$ are generated by the series expansion
	\begin{equation}\label{2.2}
		\delta^{\gamma}_{\tau, k}(\xi)=\frac{1}{\tau^{\gamma}}\sum^{\infty}_{j=0} \omega_{j}^{(\gamma, k)}\xi^{j} \quad {\rm with} \quad \delta_{\tau, k}(\xi)=\frac{1}{\tau} \sum_{j=1}^{k} \frac{1}{j} \left( 1-\xi \right)^{j},
	\end{equation}
	which can be computed by the recursion in \cite{ChenBackward2021} with the computational count $\mathcal{O}(N)$.
	
	Then ID$m$-BDF$k$ method for \eqref{rrfee} or \eqref{rfee} is designed by
	\begin{equation}\label{2.3}
		\partial^{\gamma}_{\tau, k} V^{n} - AV^{n}= \partial^{m}_{\tau, k} \left( \frac{t^{m}_{n}}{\Gamma(m+1)} A\upsilon + \frac{t_{n}^{m+1}}{\Gamma(m+2)} Ab + G^{n} \right), \quad2\leq m\leq k+1.
	\end{equation}
	\begin{remark}
		Intuitively, we require $\upsilon \in \mathcal{D}(A)$ in \eqref{2.3}.
		However, it still holds for the nonsmooth data $\upsilon\in L^2(\Omega)$, see Theorems \ref{theorem3.9}, \ref{theorem4.4}, \ref{theorema5.2} and \ref{theorem5.5}.
		Here, we mainly study on the time semi-discrete schemes \eqref{2.3}, since the spatial discretization is trivial.
		For example, we can choose $\upsilon_h=P_h\upsilon$ if $\upsilon\in L^2(\Omega)$, see \cite{ShiCorrection2022,ThomeeGalerkin2006,WangTwo2020}.
	\end{remark}
	
	\subsection{Solution representation}
	Applying Laplace transform in \eqref{rrfee}, we have
	\begin{equation*}
		\widehat{V}(z)= (z^{\gamma}-A)^{-1}\left( z^{-1} A\upsilon + z^{-2}Ab + z^{m} \widehat{G}(z) \right),\quad1<\gamma<2.
	\end{equation*}
	By inverse Laplace transform, the representation of continuous solution in \eqref{rrfee} is
	\begin{equation}\label{LT}
		V(t)= \frac{1}{2\pi i} \int_{\Gamma_{\theta, \kappa}} e^{zt} ( z^{\gamma}-A)^{-1}\left(z^{-1} A\upsilon + z^{-2}Ab + z^{m} \widehat{G}(z)\right)dz,
	\end{equation}
	where $\pi/2 < \theta < \min{(\pi, \pi/\gamma)}$, $\kappa>0$ and
	\begin{equation}\label{Gamma}
		\Gamma_{\theta, \kappa}=\{ z\in \mathbb{C}: |z|=\kappa, |\arg z|\leq \theta \} \cup \{ z\in \mathbb{C}: z=re^{\pm i\theta}, r\geq \kappa \}.
	\end{equation}

	Denote a sequence $(\kappa_n)_0^\infty$ and $\widetilde{\kappa}(\zeta)=\sum_{n=0}^{\infty}\kappa_n \zeta^n$ its generating power series.
	The representation of the discrete solution in \eqref{2.3} is obtained by the following.
	\begin{lemma}\label{Lemma2.1}
		Let $\delta^{\gamma}_{\tau, k}$ be given in \eqref{2.2} with $\gamma \in (1,2)$ and $G(t)=J^{m}g(t)$, $2\leq m\leq k+1$, $k\leq 6$ in \eqref{smoothing1}.
		Then the discrete solution of \eqref{2.3} is represented by
		\begin{equation*}
			\begin{split}
				V^{n}
				& = \frac{\tau}{2\pi i}\int_{\Gamma^{\tau}_{\theta,\kappa}} e^{zt_n} \left( \delta^{\gamma}_{\tau, k}(e^{-z\tau}) -A\right)^{-1} \delta^{m}_{\tau, k}(e^{-z\tau}) \\
				& \qquad \qquad \qquad \left( \frac{\rho_{m}(e^{-z\tau})}{\Gamma(m+1)} \tau^{m} A\upsilon + \frac{\rho_{m+1}(e^{-z\tau})}{\Gamma(m+2)} \tau^{m+1} Ab + \widetilde{G} (e^{-z\tau}) \right)dz
			\end{split}
		\end{equation*}
		with $\Gamma^{\tau}_{\theta, \kappa}=\{z\in \Gamma_{\theta, \kappa}: |\Im z|\leq \pi / \tau\}$ and $\rho_{m}(\xi)=\sum^{\infty}_{n=1}n^{m} \xi^{n} = \left(\xi\frac{d}{d\xi} \right)^{m} \frac{1}{1-\xi}$.
	\end{lemma}
	\begin{proof}
		Multiplying $\xi^{n}$ on both sides of \eqref{2.3} and summing over $n$, it yields
		\begin{equation*}
			\begin{split}
				&\sum^{\infty}_{n=1} \partial^{\gamma}_{\tau, k} V^{n} \xi^{n} - \sum^{\infty}_{n=1} AV^{n} \xi^{n} = \sum^{\infty}_{n=1} \partial^{m}_{\tau, k} \left( \frac{t^{m}_{n}}{\Gamma(m+1)} A\upsilon + \frac{t_{n}^{m+1}}{\Gamma(m+2)} Ab + G^{n}\right) \xi^{n}.
			\end{split}
		\end{equation*}
		According to  \eqref{2.1}, \eqref{2.2} and $V^0=0$, there exists
		\begin{equation*}
			\begin{split}
				\sum^{\infty}_{n=1} \partial^{\gamma}_{\tau, k} V^{n} \xi^{n}
				=& \sum^{\infty}_{n=1} \frac{1}{\tau^{\gamma}}\sum^{n}_{j=0} \omega_{j}^{(\gamma, k)} V^{n-j} \xi^{n}
				= \frac{1}{\tau^{\gamma}} \sum^{\infty}_{j=0} \omega_{j}^{(\gamma, k)} \xi^{j} \sum^{\infty}_{n=0} V^{n} \xi^{n}
				= \delta^{\gamma}_{\tau, k} (\xi) \widetilde{V} (\xi).
			\end{split}
		\end{equation*}
		Using the identities $\rho_{m}(\xi)=\sum^{\infty}_{n=1}n^{m} \xi^{n}$, $m\geq 2$ and $G^0=G(0)=0$, we obtain
		\begin{equation*}
			\begin{split}
				& \sum^{\infty}_{n=1} \partial^{m}_{\tau, k} t^{m}_{n} A\upsilon \xi^{n} = \delta^{m}_{\tau, k}(\xi) \rho_{m}(\xi) \tau^{m} A\upsilon, \quad
				\sum^{\infty}_{n=1} \partial^{m}_{\tau, k} G^{n} \xi^{n} = \delta^{m}_{\tau, k}(\xi) \widetilde{G} (\xi),
			\end{split}
		\end{equation*}
		and
		\begin{equation*}
			\begin{split}
				\sum^{\infty}_{n=1} \partial^{m}_{\tau, k} t^{m+1}_{n} Ab \xi^{n} = \delta^{m}_{\tau, k}(\xi) \rho_{m+1}(\xi) \tau^{m+1} Ab.
			\end{split}
		\end{equation*}
		
		Combining the above equations, we have
		\begin{equation}\label{ads2.17}
			\widetilde{V}(\xi) =  \left( \delta^{\gamma}_{\tau, k}(\xi) -A\right)^{-1} \delta^{m}_{\tau, k}(\xi) \left( \frac{\rho_{m}(\xi)\tau^{m}}{\Gamma(m+1)} A\upsilon  + \frac{\rho_{m+1}(\xi)}{\Gamma(m+2)} \tau^{m+1} Ab + \widetilde{G} (\xi) \right).
		\end{equation}
		From Cauchy's integral formula, the change of variables $\xi=e^{-z\tau}$ and Cauchy's theorem, we obtain the following discrete solution of \eqref{2.3}
		\begin{equation}\label{DLT}
			\begin{split}
				V^{n}
				& = \frac{\tau}{2\pi i}\int_{\Gamma^{\tau}_{\theta,\kappa}} e^{zt_n} \left( \delta^{\gamma}_{\tau, k}(e^{-z\tau}) -A\right)^{-1} \delta^{m}_{\tau, k}(e^{-z\tau}) \\
				& \qquad \qquad \qquad \left( \frac{\rho_{m}(e^{-z\tau})}{\Gamma(m+1)} \tau^{m} A\upsilon + \frac{\rho_{m+1}(e^{-z\tau})}{\Gamma(m+2)} \tau^{m+1} Ab +\widetilde{G} (e^{-z\tau}) \right)dz
			\end{split}
		\end{equation}
		with $\Gamma^{\tau}_{\theta, \kappa}=\{z\in \Gamma_{\theta, \kappa}: |\Im z|\leq \pi / \tau\}$.
		The proof is completed.
	\end{proof}
	
	Note that BDF$k$ methods are $A(\vartheta_k)$-stable with
	$\vartheta_1=\vartheta_2=90^\circ$, $\vartheta_3\approx 86.03^\circ$, $\vartheta_4\approx 73.35^\circ$, $\vartheta_5\approx 51.84^\circ$ and $\vartheta_6 \approx 17.84^\circ$;
	see \cite{Akrivisenergy2021,HairerSolving2010}.
	Thus the numerical scheme \eqref{2.3} is unconditionally stable for any $\gamma < \gamma^{\ast}(k):=\pi/(\pi-\vartheta_{k})$.
	The critical value $\gamma^{\ast}(k)$ is $1.91$, $1.68$, $1.40$ and $1.11$ for $3, 4, 5, 6$, respectively.
	In contrast, for $\gamma \geq \gamma^{\ast}(k)$, it is only conditionally stable, see \cite[p. A3137]{JinCorrection2017}.
	
	\section{Error analysis: General source function $g(t)$}\label{Se:GSF}
	In this section, we first establish the detailed error analysis for the fractional diffusion-wave equation \eqref{rfee} under the mild regularity of the general source function $g(t)$.
	\subsection{A few technical lemmas}
	We introduce some lemmas, which play an important role in convergence analysis.
	\begin{lemma}\cite{JinCorrection2017}\label{Lemma 2.3}
		Let $\delta^{\gamma}_{\tau, k}$ be given in \eqref{2.2} with $\gamma \in (1,2)$.
		For any $\epsilon>0$, there exists $\theta_{\epsilon} \in (\pi/2, \pi)$ such that for any $ \theta \in (\pi/2, \theta_{\epsilon})$ there exist the positive constants $c$, $c_{1},c_{2}$ independent of $\tau$ such that
		\begin{equation*}
			\begin{split}
				& c_{1}|z|\le |\delta_{\tau, k}(e^{-z\tau})| \le c_{2}|z|, \quad |\delta_{\tau, k}(e^{-z\tau})-z|\le c \tau^{k}|z|^{k+1}, \\
				& |\delta^{\gamma}_{\tau, k}(e^{-z\tau})-z^{\gamma}|\le c \tau^{k}|z|^{k+\gamma}, \quad \delta_{\tau, k}(e^{-z\tau}) \in \Sigma_{\pi-\vartheta_k+\epsilon},
			\end{split}
		\end{equation*}
		and
		\begin{equation*}
			\begin{split}
				& \left\| \left(\delta^{\gamma}_{\tau, k}(e^{-z\tau}) -A\right)^{-1} \right\|
				\leq c |z|^{-\gamma}, \quad \forall z\in \Gamma^{\tau}_{\theta,\kappa}.
			\end{split}
		\end{equation*}
	\end{lemma}
	
	\begin{lemma}\cite{ShiHighorder2023}\label{Lemma nn3.3}
		Let $\rho_{l}(\xi)=\sum^{\infty}_{n=1}n^{l} \xi^{n}$ with $l=0, 1, 2, \ldots, 2k$, $k\leq 6$.
		Then there exists a positive constant $c$ independent of $\tau$ such that
		\begin{equation*}
			\left| \frac{\rho_{l}(e^{-z\tau})}{\Gamma(l+1)} \tau^{l+1} - \frac{1}{z^{l+1}} \right| \leq
			\left\{ \begin{split}
				& c \tau^{l+1}, \qquad l=0,1,3,\ldots,2k-1,\\
				& c \tau^{l+2}|z| , \quad l=2,4,\ldots,2k.
			\end{split}\right.
		\end{equation*}
	\end{lemma}
	
	\begin{lemma}\cite{ShiHighorder2023}\label{Lemma 3.4}
		Let $\delta_{\tau, k}(\xi)$ be given in \eqref{2.2} and $\rho_{l}(\xi)=\sum^{\infty}_{n=1}n^{l} \xi^{n}$ with $l=0, 1, \ldots, k+m-1$, $2\leq m\leq k+1$, $k\leq 6$.
		Then there exists a positive constant $c$ independent of $\tau$ such that
		\begin{equation*}
			\left|\delta^{m}_{\tau, k}(e^{-z\tau}) \frac{\rho_{l}(e^{-z\tau})}{\Gamma(l+1)} \tau^{l+1} -\frac{z^m}{z^{l+1}} \right|
			\leq \left\{ \begin{split}
				& c \tau^{l+1} \left| z \right|^{m} + c \tau^{k}|z|^{k+m-l-1}, \qquad l=0,1,3,\ldots,\\
				& c \tau^{l+2} \left| z \right|^{m+1} + c \tau^{k}|z|^{k+m-l-1}, \quad l=2,4,\ldots.
			\end{split}\right.
		\end{equation*}
	\end{lemma}
	
	\begin{lemma}\label{Lemma 3.11}
		Let $\delta^{\gamma}_{\tau, k}$ be given in \eqref{2.2} with $\gamma \in (1,2)$ and $\rho_{l}(\xi)=\sum^{\infty}_{n=1}n^{l} \xi^{n}$ with $l=0, 1, \ldots, k+m-1$, $2\leq m\leq k+1$, $k\leq 6$.
		Then there exists a positive constant $c$ independent of $\tau$ such that
		\begin{equation*}
			\left\|K(z) \right\|\leq
			\left\{ \begin{split}
				& c \tau^{l+1} \left| z \right|^{m-\gamma} + c \tau^{k}|z|^{k+m-l-1-\gamma}, \qquad l=0,1,3,\ldots,2k-1,\\
				& c \tau^{l+2} \left| z \right|^{m+1-\gamma} + c \tau^{k}|z|^{k+m-l-1-\gamma}, \quad l=2,4,\ldots,2k
			\end{split}\right.
		\end{equation*}
		with
		\begin{equation*}
			\begin{split}
				K(z)= \left( \delta^{\gamma}_{\tau, k}(e^{-z\tau}) -A\right)^{-1} \delta^{m}_{\tau, k}(e^{-z\tau}) \frac{\rho_{l}(e^{-z\tau})}{\Gamma(l+1)} \tau^{l+1} - (z^{\gamma}-A)^{-1} \frac{z^m}{z^{l+1}}.
			\end{split}
		\end{equation*}
	\end{lemma}
	\begin{proof}
		Using Lemmas \ref{Lemma 2.3}-\ref{Lemma 3.4}, the similar arguments can be performed as \cite[Lemma 3.4]{ShiHighorder2023}, we omit it here.
	\end{proof}
	
	\begin{lemma}\label{addLemma 3.6}
		Let $\delta^{\gamma}_{\tau, k}$ be given in \eqref{2.2} with $1<\gamma <2$ and $\rho_{m}(\xi)=\sum^{\infty}_{n=1}n^{m} \xi^{n}$.
		Then there exists a positive constant $c$ independent of $\tau$ such that
		\begin{equation*}
			\left\|\mathcal{K}_\upsilon(z) \right\|\leq \left\{ \begin{split}
				& c \tau^{m+2} \left| z \right|^{m+1} + c \tau^{k}|z|^{k+m-l-1}, \quad m=2,4,6\\
				& c \tau^{m+1} \left| z \right|^{m} + c \tau^{k}|z|^{k+m-l-1}, \qquad m=3,5,7
			\end{split}\right.
		\end{equation*}
		with
		\begin{equation}\label{addeq3.1}
			\mathcal{K}_\upsilon(z)= \left(\delta^{\gamma}_{\tau, k}(e^{-z\tau}) -A\right)^{-1} \delta^{m}_{\tau, k}(e^{-z\tau}) \frac{\rho_{m}(e^{-z\tau})}{\Gamma(m+1)} \tau^{m+1} A - (z^{\gamma}-A)^{-1} z^{-1} A.
		\end{equation}
	\end{lemma}
	
	\begin{proof}
		From Lemmas \ref{Lemma 2.3}-\ref{Lemma 3.4} and Lemma 3.5 in \cite{ShiHighorder2023}, the desired result is obtained.
	\end{proof}
	
	\begin{lemma}\label{addaddLemma 3.6}
		Let $\delta^{\gamma}_{\tau, k}$ be given in \eqref{2.2} with $1<\gamma <2$ , $\rho_{m+1}(\xi)=\sum^{\infty}_{n=1}n^{m+1} \xi^{n}$.
		Then there exists a positive constant $c$ independent of $\tau$ such that
		\begin{equation*}
			\left\|\mathcal{K}_b(z) \right\|\leq \left\{ \begin{split}
				& c \tau^{m+2} \left| z \right|^{m} + c \tau^{k}|z|^{k-2}, \qquad m=2,4,6,\\
				& c \tau^{m+3} \left| z \right|^{m+1} + c \tau^{k}|z|^{k-2}, \quad m=3,5,7
			\end{split}
			\right.
		\end{equation*}
		with
		\begin{equation}\label{addaddeq3.1}
			\mathcal{K}_b(z)= \left(\delta^{\gamma}_{\tau, k}(e^{-z\tau}) -A\right)^{-1} \delta^{m}_{\tau, k}(e^{-z\tau}) \frac{\rho_{m+1}(e^{-z\tau})}{\Gamma(m+2)} \tau^{m+2} A - (z^{\gamma}-A)^{-1} z^{-2} A.
		\end{equation}
	\end{lemma}
	
	\begin{proof}
		We can check
		$(z^{\gamma}-A)^{-1} z^{-2} A = - z^{-2} + ( z^{\gamma}-A)^{-1} z^{\gamma-2}$ and
		\begin{equation*}
			\begin{split}
				& \left(\delta^{\gamma}_{\tau, k}(e^{-z\tau})-A\right)^{-1} \delta^{m}_{\tau, k}(e^{-z\tau}) \frac{\rho_{m+1}(e^{-z\tau})}{\Gamma(m+2)} \tau^{m+2} A\\
				& =\left( \delta^{\gamma}_{\tau, k}(e^{-z\tau}) -A\right)^{-1} \delta^{\gamma}_{\tau, k}(e^{-z\tau}) \delta^{m}_{\tau, k}(e^{-z\tau}) \frac{\rho_{m+1}(e^{-z\tau})}{\Gamma(m+2)} \tau^{m+2} \\
				&\quad - \delta^{m}_{\tau, k}(e^{-z\tau})\frac{\rho_{m+1}(e^{-z\tau})}{\Gamma(m+2)} \tau^{m+2}.
			\end{split}
		\end{equation*}
		Then we split $\mathcal{K}_b(z)$ into the following four parts
		\begin{equation*}
			\mathcal{K}_b(z) = J_{1} + J_{2} + J_{3} + J_{4}
		\end{equation*}
		with
		\begin{equation*}
			\begin{split}
				J_{1} =& \left(  \delta^{\gamma}_{\tau, k}(e^{-z\tau}) -A\right)^{-1}  \delta^{\gamma}_{\tau, k}(e^{-z\tau})
				\left( \delta^{m}_{\tau, k}(e^{-z\tau})\frac{\rho_{m+1}(e^{-z\tau})}{\Gamma(m+2)} \tau^{m+2} - z^{-2} \right), \\
				J_{2} =& \left( \delta^{\gamma}_{\tau, k}(e^{-z\tau}) -A\right)^{-1} \left( \delta^{\gamma}_{\tau, k}(e^{-z\tau}) - z^{\gamma} \right) z^{-2}, \\
				J_{3} =& \left( \left( \delta^{\gamma}_{\tau, k}(e^{-z\tau}) -A\right)^{-1} - ( z^{\gamma}-A)^{-1} \right)  z^{\gamma-2}, \\
				J_{4} =& z^{-2} - \delta^{m}_{\tau, k}(e^{-z\tau})\frac{\rho_{m+1}(e^{-z\tau})}{\Gamma(m+2)} \tau^{m+2}.
			\end{split}
		\end{equation*}
		According to Lemmas \ref{Lemma 2.3} and \ref{Lemma 3.4}, we estimate
		\begin{equation*}
			\left\|J_1 \right\|\leq c\left\|J_4 \right\|
			\leq \left\{\begin{split}
				& c \tau^{m+2} \left| z \right|^{m} + c \tau^{k}|z|^{k-2}, \qquad m=2,4,6,\\
				& c \tau^{m+3} \left| z \right|^{m+1} + c \tau^{k}|z|^{k-2}, \quad m=3,5,7,
			\end{split}\right.
		\end{equation*}
		and $$\left\| J_{2} \right\| \leq c |z|^{-\gamma} \tau^{k}|z|^{k+\gamma} |z|^{-2} \leq c \tau^{k} |z|^{k-2}.$$
		On the other hand, from Lemma \ref{Lemma 2.3} and
		\begin{equation*}
			\left( \delta^{\gamma}_{\tau, k}(e^{-z\tau}) -A\right)^{-1} - (z^{\gamma}-A)^{-1}
			=\left( z^{\gamma} - \delta^{\gamma}_{\tau, k}(e^{-z\tau}) \right) \left( \delta^{\gamma}_{\tau, k}(e^{-z\tau}) -A\right)^{-1} (z^{\gamma}-A)^{-1},
		\end{equation*}
		it yields
		\begin{equation*}
			\left\| J_{3} \right\| \leq c \tau^{k} |z|^{-2\gamma}  |z|^{k + 2\gamma -2} \leq c \tau^{k}  |z|^{k -2}.
		\end{equation*}
		The proof is completed.
	\end{proof}
	
	\subsection{Convergence analysis for general source function $g(t)$}
	We first provided the detailed convergence analysis for \eqref{rfee} under the mild regularity $g(t)$.
	Let  $G(t)=J^{m}g(t)$ be defined in \eqref{smoothing1}.
	By Taylor series expansion with the remainder term in integral form \cite[Eq. (3.3)]{ShiHighorder2023}, it yields
	\begin{equation}\label{gs3.3}
		G(t) = \sum_{l=0}^{k+m-2} \frac{t^l}{\Gamma(l+1)} g^{(l-m)}(0) + \frac{t^{k+m-2}}{\Gamma(k+m-1)} \ast g^{(k-1)}(t)
	\end{equation}
	with $g^{(-i)}(0)=J^{i}g(0)=0$, $i \geq 1$.
	Then we have  the following results.
	\begin{lemma}\label{lemma3.9}
		Let $V(t_{n})$ and $V^{n}$ be the solutions of \eqref{rrfee} and \eqref{2.3}, respectively.
		Let $\upsilon=b=0$ and 
		 $G(t):=\frac{t^{l}}{\Gamma(l+1)} g^{(l-m)}(0)$ with $l=0,1, \ldots, k+m-2$, $2\leq m \leq k+1$, $k\leq 6$.
		Then the following error estimate holds for any $t_n>0$
			\begin{equation*}
				\begin{split}
					& \left\| V(t_{n}) - V^{n} \right\|_{L^2(\Omega)} \\
					& \leq \left\{ \begin{split}
						& \left( c\tau^{l+1}t_{n}^{\gamma-m-1} + c\tau^{k} t_{n}^{\gamma+l-k-m} \right) \left\| g^{(l-m)}(0) \right\|_{L^2(\Omega)}, \quad l=0,1,3,\ldots,\\
						& \left( c\tau^{l+2}t_{n}^{\gamma-m-2} + c\tau^{k} t_{n}^{\gamma+l-k-m} \right) \left\| g^{(l-m)}(0) \right\|_{L^2(\Omega)}, \quad l=2,4,6,\ldots.
					\end{split}\right.
				\end{split}
			\end{equation*}
	\end{lemma}
	\begin{proof}
		This lemma can be proved in the same way as shown Lemma 3.6 of \cite{ShiHighorder2023}.
	\end{proof}
	
	\begin{lemma}\label{lemma3.10} 
		Let $V(t_{n})$ and $V^{n}$ be the solutions of \eqref{rrfee} and \eqref{2.3}, respectively.
		Let $\upsilon=b=0$, 
	 $G(t):=\frac{t^{k+m-2}}{\Gamma(k+m-1)} \ast g^{(k-1)}(t)$  and $\int_{0}^{t} (t-s)^{\gamma-2} \| g^{(k-1)}(s) \|_{L^2(\Omega)} ds <\infty$ with $1<\gamma <2$. Then the following error estimate holds for any $t_n>0$
		\begin{equation*}
			\left\|V(t_{n})-V^{n}\right\|_{L^2(\Omega)}\leq c\tau^{k} \int_{0}^{t_{n}} (t_n-s)^{\gamma-2} \left\| g^{(k-1)}(s) \right\|_{L^2(\Omega)}ds.
		\end{equation*}
	\end{lemma}
	\begin{proof}
		From \eqref{LT} with $\kappa\geq 1$ and $G(t)=\frac{t^{k+m-2}}{\Gamma(k+m-1)} \ast g^{(k-1)}(t)$, it yields
		\begin{equation}\label{nas3.6}
			V(t_{n}) =(\mathscr{E}(t)\ast G(t))(t_{n}) = \left(\left(\mathscr{E}(t)\ast \frac{t^{k+m-2}}{\Gamma(k+m-1)} \right) \ast g^{(k-1)}(t) \right)(t_{n})
		\end{equation}
		with
		\begin{equation}\label{nas3.007}
			\mathscr{E}(t)= \frac{1}{2\pi i} \int_{\Gamma_{\theta,\kappa}} e^{zt}( z^{\gamma} - A)^{-1} z^{m} dz.
		\end{equation}
		Let $\sum^{\infty}_{n=0}\mathscr{E}^{n}_{\tau}\xi^{n}=\widetilde{\mathscr{E_{\tau}}}(\xi):=\left(\delta^{\gamma}_{\tau, k}(\xi) -A\right)^{-1} \delta^{m}_{\tau, k}(\xi)$.
		Using \eqref{ads2.17}, we have
		\begin{equation*}
			\widetilde{V}(\xi)
			= \widetilde{\mathscr{E_{\tau}}}(\xi)\widetilde{G}(\xi)
			=\sum^{\infty}_{n=0}\mathscr{E}^{n}_{\tau}\xi^{n}\sum^{\infty}_{j=0}G^j\xi^{j}
			=\sum^{\infty}_{n=0}\sum^{n}_{j=0}\mathscr{E}^{n-j}_{\tau} G^j \xi^{n}=\sum^{\infty}_{n=0}V^n\xi^{n}
		\end{equation*}
		with
		\begin{equation*}
			V^{n}=\sum^{n}_{j=0}\mathscr{E}^{n-j}_{\tau} G^j:=\sum^{n}_{j=0}\mathscr{E}^{n-j}_{\tau} G(t_{j}).
		\end{equation*}
		
		By the Cauchy's integral formula, $\xi=e^{-z\tau}$ and Cauchy's theorem, one has
		\begin{equation*}
			\mathscr{E}^{n}_{\tau}
			=\frac{\tau}{2\pi i}\int_{\Gamma^{\tau}_{\theta,\kappa}} {e^{zt_n}\left( \delta^{\gamma}_{\tau, k}(e^{-z\tau}) -A\right)^{-1} \delta^{m}_{\tau, k}(e^{-z\tau}) }dz.
		\end{equation*}
		From Lemma \ref{Lemma 2.3} and $\tau t^{-1}_{n} = \frac{1}{n}\leq 1$, $\kappa\geq 1$, it means that
		\begin{equation}\label{3.0002}
			\|\mathscr{E}^{n}_{\tau}\| \leq c \tau \left( \int^{\frac{\pi}{\tau\sin\theta}}_{\kappa} e^{rt_{n}\cos\theta} r^{m-\gamma}dr +\int^{\theta}_{-\theta}e^{\kappa t_{n}\cos\psi} \kappa^{m+1-\gamma} d\psi\right)
			\leq c \tau t_{n}^{\gamma-m-1}.
		\end{equation}
		Let $ \mathscr{E}_{\tau}(t)=\sum^{\infty}_{n=0}\mathscr{E}^{n}_{\tau}\delta_{t_{n}}(t)$ and $\delta_{t_{n}}$ be the Dirac delta function at $t_{n}$.
		Then 
		\begin{equation}\label{nas3.8}
			(\mathscr{E}_{\tau}(t)\ast G(t))(t_{n}) = \left(\sum^{\infty}_{j=0}\mathscr{E}^{j}_{\tau}\delta_{t_{j}}(t) \ast G(t) \right)(t_{n})
			= \sum^{n}_{j=0}\mathscr{E}^{n-j}_{\tau} G(t_{j})=V^{n}.
		\end{equation}
		According to the above equations, we have
		\begin{equation*}
			\begin{split}
				\widetilde{(\mathscr{E}_{\tau}\ast t^{l})}(\xi)
				& = \sum^{\infty}_{n=0} \sum^{n}_{j=0}\mathscr{E}^{n-j}_{\tau}t^{l}_{j}\xi^{n}
				= \sum^{\infty}_{j=0} \sum^{\infty}_{n=0}\mathscr{E}^{n}_{\tau}t^{l}_{j}\xi^{n+j} \\
				& = \sum^{\infty}_{n=0}\mathscr{E}^{n}_{\tau}\xi^{n}\sum^{\infty}_{j=0}t^{l}_{j}\xi^{j}
				= \widetilde{\mathscr{E_{\tau}}}(\xi) \tau^{l} \sum^{\infty}_{j=0}j^{l}\xi^{j}
				= \widetilde{\mathscr{E}_{\tau}}(\xi) \tau^{l} \rho_{l}(\xi).
			\end{split}
		\end{equation*}
		Combining \eqref{nas3.6}, \eqref{nas3.8} and Lemma \ref{lemma3.9}, it leads to
		\begin{equation}\label{aadd3.8}
			\left\| \left((\mathscr{E}_{\tau}-\mathscr{E}) \ast \frac{t^{l}}{l!} \right)(t_n) \right\|\leq \left\{ \begin{split}
				&  c\tau^{l+1}t_{n}^{\gamma-m-1}  + c\tau^{k}t_{n}^{\gamma+l-k-m} , \quad l=0,1,3,\ldots\\
				&  c\tau^{l+2}t_{n}^{\gamma-m-2} + c\tau^{k}t_{n}^{\gamma+l-k-m}, \quad l=2,4,6,\ldots.
			\end{split}\right.
		\end{equation}
		with $l\leq k+m-2$.
		
		We next prove that the following estimate \eqref{3.0003} holds for any $t>0$
		\begin{equation}\label{3.0003}
			\left\|\left((\mathscr{E}_{\tau}-\mathscr{E}) \ast \frac{t^{k+m-2}}{\Gamma(k+m-1)} \right)(t)\right\| \leq  c\tau^{k}t^{\gamma-2} \quad \forall t\in (t_{n-1},t_{n}).
		\end{equation}
		Using the Taylor series expansion of $\mathscr{E}(t)$ at $t=t_{n}$, it implies
		\begin{equation*}
			\begin{split}
				\left( \mathscr{E} \ast \frac{t^{k+m-2}}{\Gamma(k+m-1)} \right)(t)
				& = \sum_{l=0}^{k+m-2}  \frac{(t-t_{n})^{l}}{\Gamma(l+1)} \left( \mathscr{E} \ast \frac{t^{k+m-l-2}}{\Gamma(k+m-l-1)} \right)(t_{n}) \\
				& \quad +\frac{1}{\Gamma(k+m-1)} \int^{t}_{t_{n}}(t-s)^{k+m-2} \mathscr{E}(s)ds,
			\end{split}
		\end{equation*}
		which  is valid for the convolution $\left( \mathscr{E}_{\tau} \ast t^{k+m-2} \right)(t)$.
		From \eqref{aadd3.8}, we get
		\begin{equation*}
			\begin{split}
				\left\|(t-t_{n})^{l} \left((\mathscr{E}_{\tau}-\mathscr{E}) \ast t^{k+m-l-2} \right)(t_n) \right\|
				& \leq c\tau^l \left( \tau^{k+m-l-1}t_{n}^{\gamma-m-1} + \tau^{k}t_{n}^{\gamma-l-2} \right)  \\
				& \leq c \tau^{k}t_{n}^{\gamma- 2} \leq  c \tau^{k}t^{\gamma- 2}.
			\end{split}
		\end{equation*}
		
		According to \eqref{nas3.007} and \eqref{fractional resolvent estimate}, the following estimate holds
		\begin{equation*}
			\| \mathscr{E}(t) \|
			\leq c \left( \int^{\infty}_{\kappa}e^{rt\cos\theta}r^{m-\gamma}dr + \int^{\theta}_{-\theta}e^{\kappa t \cos\psi}\kappa^{m+1-\gamma}d\psi \right)
			\leq c t^{\gamma-m-1}.
		\end{equation*}
		Here we use the inequality
		\begin{equation*}
			\begin{split}
				\int^{\frac{\pi}{\tau\sin\theta}}_{\kappa} e^{rt_{n}\cos\theta} r^{m-\gamma}dr \leq c t_n^{\gamma-m-1} \quad {\rm and} \quad
				\int^{\theta}_{-\theta}e^{\kappa t_{n} \cos\psi} \kappa^{m+1-\gamma} d\psi \leq c t_n^{\gamma-m-1}.
			\end{split}
		\end{equation*}
		Furthermore, we obtain
		\begin{equation*}
			\left\| \int^{t}_{t_{n}}(t-s)^{k+m-2} \mathscr{E}(s)ds \right\| \leq c \int^{t_{n}}_{t}(s-t)^{k+m-2} s^{\gamma-m-1}ds \leq c \tau^{k} t^{\gamma-2}.
		\end{equation*}
		By the definition of $ \mathscr{E}_{\tau}(t)=\sum^{\infty}_{n=0}\mathscr{E}^{n}_{\tau}\delta_{t_{n}}(t)$ in \eqref{nas3.8} and \eqref{3.0002}, the following error estimate holds for any $t \in (t_{n-1},t_{n})$
		\begin{equation*}
			\left\| \int^{t}_{t_{n}}(t-s)^{k+m-2} \mathscr{E}_{\tau}(s)ds \right\|
			\leq (t_n-t)^{k+m-2} \| \mathscr{E}^{n}_{\tau} \| \leq c \tau^{k+m-1} t_{n}^{\gamma-m-1}
			\leq c \tau^{k}t^{\gamma- 2}.
		\end{equation*}
		The desired result \eqref{3.0003} is obtained by the above inequalities.
		The proof is completed.
	\end{proof}
	
	For simplicity, we denote
	\begin{equation}\label{initialesti}
		\left\|J_{\upsilon} \right\|_{L^2(\Omega)}=
		\left\{ \begin{split}
			& c \tau^{m+2} t_{n}^{-m-2} \left\| \upsilon \right\|_{L^2(\Omega)} + c\tau^{k} t_{n}^{-k} \left\| \upsilon \right\|_{L^2(\Omega)}, \quad m=2,4,6,\\
			& c \tau^{m+1} t_{n}^{-m-1} \left\| \upsilon \right\|_{L^2(\Omega)} + c\tau^{k} t_{n}^{-k} \left\| \upsilon \right\|_{L^2(\Omega)}, \quad m=3,5,7,
		\end{split}\right.
	\end{equation}
	\begin{equation}\label{initialesti1}
		\left\|J_{b} \right\|_{L^2(\Omega)}=
		\left\{ \begin{split}
			& c \tau^{m+2} t_{n}^{-m-1} \left\| b \right\|_{L^2(\Omega)} + c\tau^{k} t_{n}^{1-k} \left\| b \right\|_{L^2(\Omega)},  \quad m=2,4,6,\\
			& c  \tau^{m+3} t_{n}^{-m-2} \left\| b \right\|_{L^2(\Omega)} + c\tau^{k} t_{n}^{1-k} \left\| b \right\|_{L^2(\Omega)}, \quad m=3,5,7,
		\end{split}\right.
	\end{equation}
	and
	\begin{equation*}
		\left\|J_{g} \right\|_{L^2(\Omega)} \! = \!
		\left\{ \begin{split}
			& c \sum\limits_{l=0}^{k-2} \left( \tau^{l+m+2}t_{n}^{\gamma-m-2} + \tau^{k}t_{n}^{\gamma+l-k} \right) \left\| g^{(l)}(0)\right\|_{L^2(\Omega)}, \quad l+m=2,4,6,\ldots,\\
			& c \sum\limits_{l=0}^{k-2} \left( \tau^{l+m+1}t_{n}^{\gamma-m-1} + \tau^{k}t_{n}^{\gamma+l-k} \right) \left\| g^{(l)}(0)\right\|_{L^2(\Omega)}, \quad l+m=3,5,7,\ldots.
		\end{split}\right.
	\end{equation*}
	Then we get the following result.
	\begin{theorem}\label{theorem3.9} 
		Let $V(t_{n})$ and $V^{n}$ be the solutions of \eqref{rrfee} and \eqref{2.3}, respectively.
		Let $\upsilon, b\in L^{2}(\Omega)$, $g\in C^{k-2}([0,T]; L^{2}(\Omega))$ and $\int_{0}^{t} (t-s)^{\gamma-2} \left\| g^{(k-1)}(s) \right\|_{L^2(\Omega)} ds <\infty$ with $1<\gamma <2$.
		Then the following error estimate holds for any $t_n>0$
		\begin{equation*}
			\begin{split}
				\left\|V^{n}-V(t_{n})\right\|_{L^2(\Omega)}
				&\leq \|J_{\upsilon}\|_{L^2(\Omega)}+\|J_{b}\|_{L^2(\Omega)}+\|J_{g}\|_{L^2(\Omega)}\\
				&\quad + c\tau^{k} \int_{0}^{t_{n}} (t_n-s)^{\gamma-2} \left\| g^{(k-1)}(s) \right\|_{L^2(\Omega)} ds.
			\end{split}
		\end{equation*}
			In particular, for $k=1$, we have the following error estimate for any $t_n>0$
			\begin{equation*}
				\left\|V^{n}-V(t_{n})\right\|_{L^2(\Omega)}
				\leq \|J_{\upsilon}\|_{L^2(\Omega)}+\|J_{b}\|_{L^2(\Omega)} + c\tau \int_{0}^{t_{n}} (t_n-s)^{\gamma-2} \left\| g(s) \right\|_{L^2(\Omega)} ds.
			\end{equation*}
	\end{theorem}
	\begin{proof}
		Subtracting \eqref{LT} from \eqref{DLT}, we have the following split
		\begin{equation*}
			V^{n}-V(t_{n}) = I_{1} - I_{2} + I_{3} - I_{4} + I_{5}
		\end{equation*}
		with the related initial terms
		\begin{equation}\label{ada3.a1}
			I_{1} = \frac{1}{2\pi i} \int_{\Gamma^{\tau}_{\theta,\kappa}} e^{zt_{n}} \mathcal{K}_ \upsilon(z) \upsilon dz, \quad
			I_{2} = \frac{1}{2\pi i} \int_{\Gamma_{\theta,\kappa}\backslash \Gamma^{\tau}_{\theta,\kappa}} e^{zt_{n}} ( z^{\gamma}-A)^{-1} z^{-1} A\upsilon dz,
		\end{equation}
		\begin{equation}\label{ada3.a2}
			I_{3} = \frac{1}{2\pi i} \int_{\Gamma^{\tau}_{\theta,\kappa}} e^{zt_{n}} \mathcal{K}_b(z)  b dz, \quad
			I_{4} = \frac{1}{2\pi i} \int_{\Gamma_{\theta,\kappa}\backslash \Gamma^{\tau}_{\theta,\kappa}} e^{zt_{n}} ( z^{\gamma}-A)^{-1} z^{-2}Ab dz.
		\end{equation}
		Here $\mathcal{K}_\upsilon(z)$ and $\mathcal{K}_b(z)$ are defined by \eqref{addeq3.1} and \eqref{addaddeq3.1}, respectively, and the related source term
		\begin{equation*}
			\begin{split}
				I_{5}
				& = \frac{\tau}{2\pi i}\int_{\Gamma^{\tau}_{\theta,\kappa}} e^{zt_n} \left( \delta^{\gamma}_{\tau, k}(e^{-z\tau}) -A\right)^{-1} \delta^{m}_{\tau,k}(e^{-z\tau}) \widetilde{G} (e^{-z\tau}) dz \\
				& \quad - \frac{1}{2\pi i} \int_{\Gamma_{\theta, \kappa}} e^{zt_n} ( z^{\gamma}-A)^{-1} z^{m} \widehat{G}(z) dz.
			\end{split}
		\end{equation*}
		
		For $I_{1}$, $I_{2}$ and $I_{5}$, the similar estimate can be performed as Theorem 3.8 in \cite{ShiHighorder2023}.
		From Lemma \ref{addaddLemma 3.6} and the resolvent estimate \eqref{fractional resolvent estimate},
		we estimate $I_{3}$ and $I_{4}$ as following
		\begin{equation}\label{addI3}
			\left\|I_{3} \right\|_{L^2(\Omega)}\leq
			\left\{ \begin{split}
				& c  \tau^{m+3} t_{n}^{-m-2} \left\| b \right\|_{L^2(\Omega)} + c\tau^{k} t_{n}^{1-k} \left\| b \right\|_{L^2(\Omega)}, \quad m=3,5,7,\\
				& c \tau^{m+2} t_{n}^{-m-1} \left\| b \right\|_{L^2(\Omega)} + c\tau^{k} t_{n}^{1-k} \left\| b \right\|_{L^2(\Omega)},  \quad m=2,4,6.
			\end{split}\right.
		\end{equation}
		and
		\begin{equation}\label{addI4}
			\left\|I_{4}\right\|_{L^2(\Omega)}
			\leq c\int_{\Gamma_{\theta,\kappa}\backslash \Gamma^{\tau}_{\theta,\kappa}} {\left|e^{zt_{n}}\right||z|^{-2} \left\| b \right\|_{L^2(\Omega)}} |dz|
			\leq c\tau^{k} t^{1-k}_{n}\left\| b \right\|_{L^2(\Omega)}.
		\end{equation}
		Then we have $\left\|I_{3} \right\|_{L^2(\Omega)}+\left\|I_{4} \right\|_{L^2(\Omega)}\leq \left\|J_{b} \right\|_{L^2(\Omega)}.$
		The proof is completed.
	\end{proof}
	
	\section{Error analysis: Hyper-singular source function $t^{\mu}q$, $-2< \mu < -1$}\label{Se:HSF}
	We first introduce the polylogarithm function as following
	\begin{equation}\label{polylogarithm function}
		Li_{p}(\xi)= \sum_{j=1}^{\infty} \frac{\xi^{j}}{j^{p}},\quad p\notin \mathbb{N}
	\end{equation}
	with the Riemann zeta function $\zeta(p)=Li_{p}(1)$.
	
	Let $g(t) = t^{\mu}q$ with $-2< \mu < -1$.
	The Laplace transform of such function does not exist in the classical sense.
	However, it can be given by the finite-part integrals.
	In this way, the Laplace transform of $g(t)$ is defined by \cite[Eq.\,(2.256)]{PodlubnyFractional1999}
	\begin{equation*}
		\widehat{g}(z) = \frac{\Gamma(\mu+1)}{z^{\mu+1}} q, \quad -2<\mu<-1.
	\end{equation*}
	Moreover, it implies $\widehat{G}(z) = \frac{\Gamma(\mu+1)}{z^{\mu+m+1}} q$.
	Here $G(t)=J^{m}g(t) = \frac{\Gamma(\mu+1)t^{\mu+m}}{\Gamma(\mu+m+1)}q$ is calculated by the Hadamard finite-part integrals in \eqref{smoothing1}.
	
	According to \eqref{LT} and \eqref{DLT}, we obtain the continuous solution
	\begin{equation*}
		V(t) = \frac{1}{2\pi i} \int_{\Gamma_{\theta, \kappa}} e^{zt} ( z^{\gamma}-A)^{-1}\left( z^{-1} A\upsilon + z^{-2}Ab + \frac{\Gamma(\mu+1)}{z^{\mu+1}} q \right) dz,
	\end{equation*}
	and the discrete solution
	\begin{equation*}
		\begin{split}
			V^{n}
			&=\frac{\tau}{2\pi i}\int_{\Gamma^{\tau}_{\theta,\kappa}} e^{zt_n} \left(  \delta^{\gamma}_{\tau, k}(e^{-z\tau}) -A\right)^{-1} \delta^{m}_{\tau, k}(e^{-z\tau}) \\
			& \qquad \qquad \quad \left( \frac{\rho_{m}(e^{-z\tau})}{\Gamma(m+1)} \tau^{m} A\upsilon + \frac{\rho_{m+1}(e^{-z\tau})}{\Gamma(m+2)} \tau^{m+1} Ab + \widetilde{G} (e^{-z\tau}) \right)dz
		\end{split}
	\end{equation*}
	with
	\begin{equation*}
		\begin{split}
			\widetilde{G}(\xi)= \sum^{\infty}_{n=1} G^{n} \xi^{n}
			&= q \frac{ \Gamma(\mu+1) \tau^{\mu+m} \sum^{\infty}_{n=1} n^{\mu+m}\xi^{n}}{\Gamma(\mu+m+1)}
			= q \frac{ \Gamma(\mu+1) \tau^{\mu+m} Li_{-\mu-m}(\xi)}{\Gamma(\mu+m+1)}.
		\end{split}
	\end{equation*}
	
	\begin{lemma}\cite{Jinanalysis2016}\label{addLemma:LipCA}
		Let $|z\tau|\leq \frac{\pi}{\sin\theta}$ and $\theta>\pi/2$ be close to $\pi/2$.
		Then we have
		\begin{equation*}
			Li_{p}(e^{-z\tau}) = \Gamma(1-p)(z\tau)^{p-1} + \sum_{j=0}^{\infty} (-1)^{j} \zeta(p-j)\frac{(z\tau)^{j}}{\Gamma(j+1)}, \quad p\notin \mathbb{N},
		\end{equation*}
		and the infinite series converges absolutely.
		Here $\zeta$ denotes the Riemann zeta function.
	\end{lemma}
	\begin{lemma}\cite{ShiHighorder2023}\label{addLemma4.5}
		Let $\rho_{\mu+m}(\xi)=\sum^{\infty}_{n=1}n^{\mu+m} \xi^{n}$
		with $2\leq m\leq k+1$, $k \leq 6$. Then there exists a positive constant $c$ independent of $\tau$ such that
		\begin{equation*}
			\left| \frac{\rho_{\mu+m}(e^{-z\tau})}{\Gamma{(\mu+m+1)}} \tau^{\mu+m+1} - \frac{1}{z^{\mu+m+1}} \right| \leq c \tau^{\mu+m+1}, \quad -2< \mu < -1.
		\end{equation*}
	\end{lemma}
	
	\begin{lemma}\label{addLemma 4.3}
		Let $\delta^{\gamma}_{\tau, k}(\xi)$ be given in \eqref{2.2} with $1<\gamma <2$, $k \leq 6$.
		Then there exists a positive constant $c$ independent of $\tau$ such that for $ \forall z\in \Gamma^{\tau}_{\theta,\kappa}$
		\begin{equation*}
			\left\| \left( \delta^{\gamma}_{\tau, k}(e^{-z\tau}) -A\right)^{-1} \delta^{m}_{\tau, k}(e^{-z\tau}) - ( z^{\gamma}-A)^{-1} z^{m} \right\|
			\leq c \tau^{k} |z|^{k+m-\gamma}.
		\end{equation*}
	\end{lemma}
	\begin{proof}
		Since
		\begin{equation*}
			\left( \delta^{\gamma}_{\tau, k}(e^{-z\tau}) -A\right)^{-1} \delta^{m}_{\tau, k}(e^{-z\tau})  - ( z^{\gamma}-A)^{-1} z^{m} = I + II
		\end{equation*}
		with
		\begin{equation*}
			\begin{split}
				I  & = \left( \delta^{\gamma}_{\tau, k}(e^{-z\tau}) -A\right)^{-1} \left( \delta^{m}_{\tau, k}(e^{-z\tau}) - z^{m} \right), \\
				II & = \left( \left( \delta^{\gamma}_{\tau, k}(e^{-z\tau}) -A\right)^{-1} - (z^{\gamma}-A)^{-1} \right) z^{m}.
			\end{split}
		\end{equation*}
		Using Lemmas \ref{Lemma 2.3} and \ref{Lemma 3.11}, we estimate $I$ and $II$ respectively as following
		\begin{equation*}
			\|I\| \leq c \tau^{k}  |z|^{k+m-\gamma} \quad {\rm and} \quad \|II\| \leq c \tau^{k} |z|^{k + \gamma} |z|^{-\gamma} |z|^{-\gamma}  |z|^{m} = c \tau^{k}  |z|^{k+m-\gamma}.
		\end{equation*}
		By the triangle inequality, the proof is completed.
	\end{proof}
	
	\begin{theorem}\label{theorem4.4}
		Let $V(t_{n})$ and $V^{n}$ be the solutions of \eqref{rrfee} and \eqref{2.3}, respectively.
		Let $\upsilon, b \in L^{2}(\Omega)$ and $g(x,t)=t^{\mu}q$, $-2< \mu < -1$, $q \in L^{2}(\Omega)$.
		Then the following error estimate holds for any $t_n>0$
		\begin{equation*}
			\begin{split}
				\left\|V^{n}-V(t_{n})\right\|_{L^2(\Omega)}
				& \leq \left\|J_{\upsilon} \right\|_{L^2(\Omega)} + \left\|J_{b} \right\|_{L^2(\Omega)} \\
				& \quad + c \left(\tau^{\mu + m+1} t^{\gamma-m-1}_{n} + \tau^{k} t_{n}^{\gamma+\mu-k} \right)\| q \|_{L^2(\Omega)}
			\end{split}
		\end{equation*}
		with $\left\|J_{\upsilon} \right\|_{L^2(\Omega)}$ and $\left\|J_{b} \right\|_{L^2(\Omega)}$ in \eqref{initialesti} and \eqref{initialesti1}.
	\end{theorem}
	\begin{proof}
		Subtracting \eqref{LT} from \eqref{DLT}, we obtain
		\begin{equation*}
			V^{n}-V(t_{n})=I_{1}-I_{2} + I_{3} - I_{4} + I_{5} - I_{6},
		\end{equation*}
		where $I_{1}$, $I_{2}$ and $I_{3}$, $I_{4}$ respectively, are defined in \eqref{ada3.a1} and \eqref{ada3.a2}, and
		\begin{equation*}
			\begin{split}
				I_{5} =& \frac{1}{2\pi i}\int_{\Gamma^{\tau}_{\theta,\kappa}} e^{zt_{n}} \! \left[ \left( \delta^{\gamma}_{\tau, k}(e^{-z\tau})-A \right)^{-1} \delta^{m}_{\tau, k}(e^{-z\tau}) \tau \widetilde{G} (e^{-z\tau}) -(z^{\gamma}-A)^{-1} z^{m} \widehat{G}(z) \right] dz,\\
				I_{6} =& \frac{1}{2\pi i} \int_{\Gamma_{\theta,\kappa}\backslash \Gamma^{\tau}_{\theta,\kappa}} e^{zt_n} (z^{\gamma}-A)^{-1} z^{m} \widehat{G}(z) dz.
			\end{split}
		\end{equation*}
		According to \eqref{addI3} and \eqref{addI4}, we estimate
		\begin{equation*}
			\left\|I_{3} \right\|_{L^2(\Omega)}+\left\|I_{4} \right\|_{L^2(\Omega)}\leq \left\|J_{b} \right\|_{L^2(\Omega)}.
		\end{equation*}
		For $I_{1}$, $I_{2}$ and $I_{5}$, $I_{6}$, the similar estimate can be performed as in Theorem 4.4 of \cite{ShiHighorder2023}.
		The proof is completed.
	\end{proof}
	
	\section{Error analysis: Source function $t^{\mu}\circ f(t)$ with $-2< \mu < -1$}\label{Se:SFC}
	We next analyze the error estimate for the fractional evolution equation \eqref{fee} with the hyper-singular source function $t^{\mu}\circ f(t)$ based on Section \ref{Se:GSF} and \ref{Se:HSF}.
	\subsection{Convergence analysis: Convolution source function $t^{\mu} \circledast f(t)$, $-2< \mu < -1$}
	Let $f(t) = \sum_{j=0}^{k-1} \frac{t^{j}}{\Gamma(j+1)}f^{(j)}(0) + \frac{t^{k-1}}{\Gamma(k)} \ast f^{(k)}(t)$, $1\leq k\leq 6$.
	Then we obtain
	\begin{equation*}
		g(t) = t^{\mu} \circledast f(t) = \sum_{j=0}^{k-1} \frac{ \Gamma(\mu+1)t^{\mu+j+1}}{\Gamma(\mu+j+2)}f^{(j)}(0) + t^{\mu} \circledast \left( \frac{t^{k-1}}{\Gamma(k)} \ast f^{(k)}(t) \right).
	\end{equation*}
	Moreover, for $2\leq m\leq k+1$, $k\leq 6$, it yields
	\begin{equation*}
		G(t) =J^{m}g(t)= \sum_{j=0}^{k-1} \frac{ \Gamma(\mu+1)t^{\mu+j+m+1}}{\Gamma(\mu+j+m+2)}f^{(j)}(0) + \frac{t^{k+m-1}}{\Gamma(k+m)} \ast \left( t^{\mu} \circledast f^{(k)}(t) \right).
	\end{equation*}
	
	\begin{lemma}\label{addLemma4.6}
		Let $V(t_{n})$ and $V^{n}$ be the solutions of \eqref{rrfee} and \eqref{2.3}, respectively.
		Let $\upsilon=b=0$, $g(t)= t^{\mu} \circledast \frac{t^{k-1}}{\Gamma(k)} \ast f^{(k)}(t)$ and $\int_{0}^{t} (t - s)^{\gamma+\mu} \left\| f^{(k)}(s) \right\|_{L^2(\Omega)}ds<\infty $ with $-2< \mu < -1$, $1<\gamma<2$.
		Then the following error estimate holds for any $t_n>0$
		\begin{equation*}
			\left\|V(t_{n})-V^{n}\right\|_{L^2(\Omega)} \leq c\tau^{k} \int_{0}^{t_{n}} (t_{n} - s)^{\gamma+\mu} \left\| f^{(k)}(s) \right\|_{L^2(\Omega)}ds .
		\end{equation*}
	\end{lemma}
	\begin{proof}
		Using Hadamard’s finite-part integral of \eqref{Hadmint} and \eqref{smoothing1}, we have
		\begin{equation*}
			g(t)= t^{\mu} \circledast \frac{t^{k-1}}{\Gamma(k)} \ast f^{(k)}(t)
			= \frac{t^{k-2}}{\Gamma(k-1)} \ast \frac{t^{\mu+1}}{\mu+1} \ast f^{(k)}(t), \quad k\geq 2
		\end{equation*}
		and
		\begin{equation*}
			g(t)= t^{\mu} \circledast \frac{t^{k-1}}{\Gamma(k)} \ast f^{(k)}(t) = \frac{ t^{\mu+1}}{\mu+1} \ast f^{(k)}(t), \quad k=1.
		\end{equation*}
Thus,  for $2\leq m\leq k+1$, $k\leq 6$, we can check 	
		\begin{equation*}
			G(t) = J^{m}g(t) =\frac{t^{k+m-2}}{\Gamma(k+m-1)} \ast g^{(k-1)}(t)
			~~{\rm with}~~g^{(k-1)}(t) = \frac{t^{\mu+1}}{\mu+1} \ast f^{(k)}(t),~ k\geq 1.
		\end{equation*}

		From Lemma \ref{lemma3.10}, it yields 
		\begin{equation*}
			\begin{split}
				\left\|V(t_{n})-V^{n}\right\|_{L^2(\Omega)}
				&\leq c\tau^{k} \int_{0}^{t_{n}} (t_n-s)^{\gamma-2} \left\| s^{\mu+1} \ast f^{(k)}(s) \right\|_{L^2(\Omega)}ds \\
				& \leq c\tau^{k} \int_{0}^{t_{n}} (t_{n} - s)^{\gamma-2} s^{\mu+1} \ast \left\| f^{(k)}(s) \right\|_{L^2(\Omega)}ds\\
				& \leq c\tau^{k} \int_{0}^{t_{n}} (t_{n} - s)^{\gamma+\mu} \left\| f^{(k)}(s) \right\|_{L^2(\Omega)}ds.
			\end{split}
		\end{equation*}
		The proof is completed.
	\end{proof}
	\begin{theorem}\label{theorema5.2}
		Let $V(t_{n})$ and $V^{n}$ be the solutions of \eqref{rrfee} and \eqref{2.3}, respectively.
		Let $\upsilon, b \in L^{2}(\Omega)$, $g(t)=t^{\mu} \circledast f(t) $, $-2 < \mu < -1$ and $f \in C^{k-1}([0,T]; L^{2}(\Omega))$, $\int_{0}^{t} (t - s)^{\gamma+\mu} \left\| f^{(k)}(s) \right\|_{L^2(\Omega)}ds<\infty $.
		Then the error estimate holds for any $t_n>0$
		\begin{equation*}
			\begin{split}
				\left\|V^{n}-V(t_{n})\right\|_{L^2(\Omega)}
				&\leq \left\|J_{\upsilon} \right\|_{L^2(\Omega)} + \left\|J_{b} \right\|_{L^2(\Omega)} + c\tau^{k} \int_{0}^{t_{n}} (t_{n} - s)^{\gamma+\mu} \left\| f^{(k)}(s) \right\|_{L^2(\Omega)}ds\\
				& \quad + c \sum_{j=0}^{k-1}\left( \tau^{\mu + j + m+2} t^{\gamma-m-1}_{n}  + \tau^{k} t_{n}^{\gamma +\mu+j +1-k} \right) \left\| f^{(j)}(0) \right\|_{L^2(\Omega)}
			\end{split}
		\end{equation*}
		with $\left\|J_{\upsilon} \right\|_{L^2(\Omega)}$ and $\left\|J_{b} \right\|_{L^2(\Omega)}$ in \eqref{initialesti} and \eqref{initialesti1}.
	\end{theorem}
	\begin{proof}
		From Theorem \ref{theorem4.4} and Lemma \ref{addLemma4.6}, the desired result is obtained.
	\end{proof}
	
	\subsection{Convergence analysis: Product source function $t^{\mu} f(t)$, $-2< \mu < -1$}
	Let  $f(t) = \sum_{j=0}^{k} \frac{t^{j}}{\Gamma(j+1)}f^{(j)}(0) + \frac{t^{k}}{\Gamma(k+1)} \ast f^{(k+1)}(t)$.
	Then we get 
	\begin{equation*}
		\begin{split}
		G(t)
		& = J^{m}g(t)=J^m\left(t^{\mu} f(t)\right)\\
		&= \sum_{j=0}^{k} \frac{ \Gamma(\mu+j+1)t^{\mu+j+m}}{\Gamma(\mu+j+m+1)\Gamma(j+1)}f^{(j)}(0) + \frac{t^{m-1}}{\Gamma(m)} \ast \left( t^{\mu} \left( \frac{t^{k}}{\Gamma(k+1)} \ast f^{(k+1)}(t) \right)\right).
		\end{split}
	\end{equation*}
	
	\begin{lemma}\label{lem5.3ad}
	Let $h_{k}(t)= t^{\mu} \left( \frac{t^{k}}{\Gamma(k+1)} \ast f^{(k+1)}(t) \right)$ with $-2< \! \mu \! < -1 $, $1 \leq \!k\! \leq 6$ and
	\begin{equation*}
		f \in C^{k}([0,T]; L^{2}(\Omega)), \quad \int_{0}^{t} (t-s)^{\mu+1} \left\| f^{(k+1)}(s) \right\|_{L^2(\Omega)}ds < \infty.
	\end{equation*}
	Then the following error estimate holds
	\begin{equation*}
		\left\| h_{k}^{(l)}(t) \right\|_{L^2(\Omega)} \leq c \int_{0}^{t} (t-s)^{\mu+k-l}\left\| f^{(k+1)}(s) \right\|_{L^2(\Omega)} ds \quad \forall l \leq k-1, ~1 \leq k \leq 6.
	\end{equation*}
	In particular, $h_{k}^{(l)}(0)=0$ with $l\leq k-2$, $2 \leq k \leq 6$.
	\end{lemma}
	\begin{proof}
	According to the Leibnitz's formula for $h_{k}(t)$, it yields
	\begin{equation}\label{addeq5.5}
		\begin{split}
			h_{k}^{(l)}(t) = \sum_{j=0}^{l} \binom{l}{j} \frac{\Gamma(\mu+1)}{\Gamma(\mu+1-j)} t^{\mu-j} \left( \frac{t^{k-l+j}}{\Gamma(k+1-l+j)} \ast f^{(k+1)}(t) \right) ~\forall l\leq k-1.
		\end{split}
	\end{equation}
	Using \eqref{addeq5.5}, we have the following estimate
	\begin{equation*}
		\begin{split}
			\left\| h_{k}^{(l)}(t) \right\|_{L^2(\Omega)}
			&\leq c\sum_{j=0}^{l} t^{\mu-j} \int_{0}^{t} (t-s)^{k-l+j}\left\| f^{(k+1)}(s) \right\|_{L^2(\Omega)} ds\\
			&\leq c \int_{0}^{t} (t-s)^{\mu+k-l}\left\| f^{(k+1)}(s) \right\|_{L^2(\Omega)} ds, \quad l\leq k-1,~ 1 \leq k \leq 6.
		\end{split}
	\end{equation*}
	In particular, for $l\leq k-2$, $2 \leq k \leq 6$, we have $\left\| h_{k}^{(l)}(t) \right\|_{L^2(\Omega)} \leq c t^{\mu+k-l}$, which implies $h_{k}^{(l)}(0)=0$.
	The proof is completed.
	\end{proof}
	
	\begin{lemma}\label{Lemma5.7}
		Let $V(t_{n})$ and $V^{n}$ be the solutions of \eqref{rrfee} and \eqref{2.3}, respectively.
		Let $\upsilon=b=0$, $g(t)=  t^{\mu} \left( \frac{t^{k}}{\Gamma(k+1)} \ast f^{(k+1)}(t) \right) $, $-2< \! \mu< \!-1$ and $f\in C^{k}([0,T]; L^{2}(\Omega))$, $\int_{0}^{t}(t-s)^{\gamma+ \mu} \left\| f^{(k+1)}(s) \right\|_{L^2(\Omega)}ds <\infty$.
		Then the following error estimate holds for any $t_n>0$
		\begin{equation*}
			\left\|V(t_{n}) - V^{n}\right\|_{L^2(\Omega)} \leq c \tau^{k} \int_{0}^{t_{n}}(t_n-s)^{\gamma+\mu} \left\| f^{(k+1)}(s)\right\|_{L^2(\Omega)} ds.
		\end{equation*}
	\end{lemma}
	\begin{proof}
		Let $g(t):=h_k(t)= t^{\mu} \left( \frac{t^{k}}{\Gamma(k+1)} \ast f^{(k+1)}(t) \right)$.
		From  Lemma \ref{lem5.3ad} and using the Taylor series expansion of $h_k(t)$ at $t=0$ with $h^{(l)}(0)=0$, $l\leq k-2$, we have
		\begin{equation*}
			G(t)=J^mg(t)= \frac{t^{m-1}}{\Gamma(m)} \ast h_k(t) = \frac{t^{k+m-2}}{\Gamma(k+m-1)} \ast h_k^{(k-1)}(t),~2\leq m\leq k+1, k \leq 6.
		\end{equation*}
Here $h_k^{(k-1)}(t)$ is defined by \eqref{addeq5.5}.
		According to  Lemmas \ref{lemma3.10} and \ref{lem5.3ad}, it yields
		\begin{equation*}
			\begin{split}
				\left\|V(t_{n})\!-\!V^{n}\right\|_{L^2(\Omega)}
				& \leq c \tau^{k} \int_{0}^{t_{n}} (t_n-s)^{\gamma-2} \left\| h_k^{(k-1)}(s) \right\|_{L^2(\Omega)} ds \\
				& \leq c \tau^{k} \int_{0}^{t_{n}} (t_n-s)^{\gamma-2}  \int_{0}^{s}  (s-w)^{\mu+1} \left\| f^{(k+1)}(w) \right\|_{L^2(\Omega)} dw ds \\
				& \leq c \tau^{k} \int_{0}^{t_{n}}(t_n-s)^{\gamma+ \mu} \left\| f^{(k+1)}(s) \right\|_{L^2(\Omega)}ds.
			\end{split}
		\end{equation*}
		The proof is completed.
	\end{proof}
	
	\begin{theorem}\label{theorem5.5}
		Let $V(t_{n})$ and $V^{n}$ be the solutions of \eqref{rrfee} and \eqref{2.3}, respectively.
		Let $\upsilon, b \in L^{2}(\Omega)$, $g(t)= t^{\mu} f(t)$ with $-2< \mu < -1$ and $f\in C^{k}([0,T]; L^{2}(\Omega))$, $\int_{0}^{t } (t -s)^{\gamma+\mu} \left\| f^{(k+1)}(s) \right\|_{L^2(\Omega)} ds <\infty$.
		Then the following error estimate holds for any $t_n>0$
		\begin{equation*}
			\begin{split}
				\left\|V^{n}-V(t_{n})\right\|_{L^2(\Omega)}
				& \leq \left\|J_{\upsilon} \right\|_{L^2(\Omega)} \! + \left\|J_{b} \right\|_{L^2(\Omega)} + c \tau^{k} \int_{0}^{t_{n}} \! (t_{n} \! -s)^{\gamma+\mu} \left\| f^{(k+1)}(s) \right\|_{L^2(\Omega)} ds  \\
				& \quad + \sum_{j=0}^{k} \left( c\tau^{\mu + j + m + 1} t_{n}^{\gamma - m - 1} + c\tau^{k} t_{n}^{\gamma+ \mu + j -k} \right) \left\| f^{(j)}(0) \right\|_{L^2(\Omega)} \\
			\end{split}
		\end{equation*}
		with $\left\|J_{\upsilon} \right\|_{L^2(\Omega)}$ and $\left\|J_{b} \right\|_{L^2(\Omega)}$ in \eqref{initialesti} and \eqref{initialesti1}.
	\end{theorem}
	\begin{proof}
		From Theorems \ref{theorem3.9}, \ref{theorem4.4} and Lemma \ref{Lemma5.7}, the desired result is obtained.
	\end{proof}
	
	\section{Convergence analysis for subdiffusion model}\label{Se:Sub}
	Consider the subdiffusion model with the hyper-singular source term \cite[Eq.(4.2.57)]{KilbasTheory2006}, whose prototype equation is, for $0<\gamma<1$
	\begin{equation} \label{subfee}
		\partial^{\gamma}_t (u(t)- \upsilon) - A u(t)=g(t):=t^{\mu}\circ f(t), \quad -2<\mu<-1
	\end{equation}
	with the initial condition $u(0)=\upsilon$.
	Let $V(t)=u(t)-\upsilon$ with $V(0)=0$.
	Then model \eqref{subfee} can be rewritten as
	\begin{equation}\label{rsubfee}
		\partial^{\gamma}_{t} V(t) - A V(t)=\partial^{m}_{t} \left( \frac{t^{m}}{\Gamma(m+1)} A\upsilon  + G(t)\right), \quad 0<t\leq T.
	\end{equation}
	Then ID$m$-BDF$k$ method for \eqref{rsubfee} is designed by
	\begin{equation}\label{sub2.3}
		\partial^{\gamma}_{\tau, k} V^{n} - AV^{n}= \partial^{m}_{\tau, k} \left( \frac{t^{m}_{n}}{\Gamma(m+1)} A\upsilon  + G^{n} \right),\quad 2\leq m\leq k+1.
	\end{equation}
	
	Using the same argument as in the proof of Theorems \ref{theorem4.4}, \ref{theorema5.2} and \ref{theorem5.5}, we can easily carry out the proof of Theorems \ref{theorem6.1}-\ref{theorem6.3} below.
	\begin{theorem}\label{theorem6.1}
		Let $V(t_{n})$ and $V^{n}$ be the solutions of \eqref{rsubfee} and \eqref{sub2.3}, respectively.
		Let $\upsilon \in L^{2}(\Omega)$ and $g(x,t)=t^{\mu}q $, $-2< \mu < -1$, $q \in L^{2}(\Omega)$.
		Then the following error estimate holds for any $t_n>0$
		\begin{equation*}
			\left\|V^{n}-V(t_{n})\right\|_{L^2(\Omega)}
			\leq \left\|J_{\upsilon} \right\|_{L^2(\Omega)} + c \left(\tau^{\mu + m+1} t^{\gamma-m-1}_{n} + \tau^{k} t_{n}^{\gamma+\mu-k} \right)\| q \|_{L^2(\Omega)}
		\end{equation*}
		with $\left\|J_{\upsilon} \right\|_{L^2(\Omega)}$ in \eqref{initialesti}.
	\end{theorem}
	
	\begin{theorem}\label{theorem6.2}
		Let $V(t_{n})$ and $V^{n}$ be the solutions of \eqref{rsubfee} and \eqref{sub2.3}, respectively.
		Let $\upsilon \in L^{2}(\Omega)$, $g(t)=t^{\mu} \circledast f(t) $, $-2 < \mu < -1$ and $f \in C^{k}([0,T]; L^{2}(\Omega))$, $\int_{0}^{t } (t -s)^{\gamma+\mu+1} \left\| f^{(k+1)}(s) \right\|_{L^2(\Omega)} ds <\infty$.
		Then the error estimate holds for any $t_n>0$
		\begin{equation*}
			\begin{split}
				\left\|V^{n}-V(t_{n})\right\|_{L^2(\Omega)}
				&\leq \left\|J_{\upsilon} \right\|_{L^2(\Omega)} + c\tau^{k} \int_{0}^{t_{n}} (t_{n} - s)^{\gamma+\mu+1} \left\| f^{(k+1)}(s) \right\|_{L^2(\Omega)}ds\\
				& \quad + c \sum_{j=0}^{k}\left( \tau^{\mu + j + m+2} t^{\gamma-m-1}_{n}  + \tau^{k} t_{n}^{\gamma +\mu+j +1-k} \right) \left\| f^{(j)}(0) \right\|_{L^2(\Omega)} \\
			\end{split}
		\end{equation*}
		with $\left\|J_{\upsilon} \right\|_{L^2(\Omega)}$ in \eqref{initialesti}.
	\end{theorem}
	
	\begin{theorem}\label{theorem6.3}
		Let $V(t_{n})$ and $V^{n}$ be the solutions of \eqref{rsubfee} and \eqref{sub2.3}, respectively.
		Let $\upsilon \in L^{2}(\Omega)$, $g(t)= t^{\mu} f(t)$ with $-2< \mu < -1$ and $f\in C^{k}([0,T]; L^{2}(\Omega))$, $\int_{0}^{t } (t -s)^{\gamma+\mu+1} \left\| f^{(k+1)}(s) \right\|_{L^2(\Omega)} ds <\infty$.
		Then the following error estimate holds for any $t_n>0$
		\begin{equation*}
			\begin{split}
				\left\|V^{n}-V(t_{n})\right\|_{L^2(\Omega)}
				& \leq \left\|J_{\upsilon}\right\|_{L^2(\Omega)} + c\tau^{k} \int_{0}^{t_{n}}(t_{n} -s)^{\gamma+\mu+1} \left\|f^{(k+1)}(s)\right\|_{L^2(\Omega)} ds \\
				& \quad + \sum_{j=0}^{k} \left( c\tau^{\mu + j + m + 1} t_{n}^{\gamma - m - 1} + c\tau^{k} t_{n}^{\gamma+ \mu + j -k} \right) \left\| f^{(j)}(0) \right\|_{L^2(\Omega)} \\
			\end{split}
		\end{equation*}
		with $\left\|J_{\upsilon} \right\|_{L^2(\Omega)}$ in \eqref{initialesti}.
	\end{theorem}
	
	\section{Numerical results}\label{Se:numer}
	In this section, we illustrate the convergence analysis of the presented schemes on several examples.
	The numerical errors are measured by the discrete $L^2$-norm ($||\cdot||_{l_2}$) at the terminal time.
	The space direction is discretized by the spectral method with Chebyshev-Gauss-Lobatto points \cite{ShenSpectral2011}.
	Since the solution is unknown, the convergence order of the numerical experiments are computed by
	\begin{equation*}
		{\rm Convergence ~Order}=\frac{\ln \left(||u^{N/2}-u^{N}||_{l_2}/||u^{N}-u^{2N}||_{l_2}\right)}{\ln 2}.
	\end{equation*}
	
	For the sake of brevity and readability, we mainly focus on ID$m$-BDF$6$ method of  \eqref{2.3} and \eqref{sub2.3} for simulating the fractional diffusion-wave equation \eqref{fee} and subdiffusion equation \eqref{subfee}, respectively.
	For other cases, such as ID$m$-BDF$k$ with $2\leq m\leq k+1$, $1\leq k <6$, the numerical experiments can be similar performed.
	
	\subsection{Fractional diffusion-wave equation}
	Let $T=1$ and $\Omega=(-1,1)$.
	Let us consider the following two examples in \eqref{fee} or \eqref{rrfee}:
	\begin{description}
		\item[(a)] $g(x,t)= 0$;
		\item[(b)] $g(x,t)= (1+t^{\mu}) \circ e^te^x \left(1+\chi_{\left(0,1\right)}\left(x\right)\right),~-2<\mu<-1$
	\end{description}
	with the initial values $\upsilon(x)=\sin(x)\sqrt{1-x^2}$ and $b(x)=\cos(x)\sqrt{1-x^2}$.
	
	Here $J^1g(x,t)=1\circledast g(t)$ is calculated by the Hadamard finite-part integrals, e.g.,
	\begin{equation*}
		J^1 (t^\mu e^t)\!=\! \oint_0^t s^\mu e^sds
		\!=\!\frac{1}{\mu+1}\left[ t^{\mu+1} e^t\!-\!\int_0^ts^{\mu+1} e^sds\right]
		\!=\!\frac{1}{\mu+1}\left[ t^{\mu+1} e^t-1\ast (t^{\mu+1} e^t)\right].
	\end{equation*}
	Note that $J$ may map a hyper-singular function $t^\mu e^t$ to a weakly singular function $t^{\mu+1} e^t$, $-2<\mu<-1$.
	Repeating the $m$-fold integral operator for $g(t)$, one has
	\begin{equation}
		G(t)=J^{m}g(t)=\frac{t^{m-2}}{\Gamma(m-1)}\ast(1\circledast g(t)),~m\geq 2,
	\end{equation}
	which are computed by JacobiGL Algorithm \cite{ChenHigh2015,HesthavenNodal2008}.
	
\begin{table}[htbp]
	{\footnotesize \begin{center}
			\caption{Case {\bf(a)}: convergent order of ID$m$-BDF$6$.}
			\begin{tabular}{|l|l| l l l l l l|} \hline
				$\gamma$ &$m$ &  $N=200$    & $N=400$     & $N=800$     &  $N=1600$    & $N=3200$     & Rate        \\ \hline
				\multirow{6}{*}{$1.3$}
				& 2    & 3.1057e-07  & 1.5956e-12  & 1.0160e-13  & 6.3778e-15  & 3.9903e-16  & 3.9984     \\
				& 3    & 1.8491e-08  & 5.6217e-13  & 3.4207e-14  & 2.1250e-15  & 1.3262e-16  & 4.0020     \\
				& 4    & 4.3056e-08  & 1.5234e-14  & 2.3547e-16  & 3.6718e-18  & 5.7362e-20  & 6.0002     \\
				& 5    & 2.7576e-08  & 2.7285e-14  & 4.3975e-16  & 6.9764e-18  & 1.0983e-19  & 5.9891     \\
				& 6    & 8.2645e-09  & 4.5791e-14  & 7.3039e-16  & 1.1526e-17  & 1.8097e-19  & 5.9929     \\
				& 7    & 1.9552e-11  & 6.5498e-14  & 1.0372e-15  & 1.6312e-17  & 2.5567e-19  & 5.9954     \\ \hline
				\multirow{6}{*}{$1.7$}
				& 2    & 6.4264e-10  & 3.4917e-11  & 2.1408e-12  & 1.3435e-13  & 8.4136e-15  & 3.9971     \\
				& 3    & 4.8966e-10  & 1.7329e-11  & 7.8237e-13  & 4.5680e-14  & 2.8177e-15  & 4.0189     \\
				& 4    & 4.6720e-10  & 1.1633e-11  & 2.0924e-13  & 3.4568e-15  & 5.5368e-17  & 5.9642     \\
				& 5    & 5.2989e-10  & 1.2468e-11  & 2.2141e-13  & 3.6406e-15  & 5.8192e-17  & 5.9672     \\
				& 6    & 5.9812e-10  & 1.3371e-11  & 2.3454e-13  & 3.8386e-15  & 6.1232e-17  & 5.9701     \\
				& 7    & 6.7027e-10  & 1.4331e-11  & 2.4850e-13  & 4.0491e-15  & 6.4463e-17  & 5.9729     \\ \hline
			\end{tabular}\label{table:11}
	\end{center}}
\end{table}
	
\begin{table}[htbp]
	{\footnotesize \begin{center}
			\caption{Case {\bf(b)} with convolution: convergent order of ID$m$-BDF$6$.}
			\begin{tabular}{|l|l| l l l l l l|} \hline
				$(\gamma,\mu)$ &$m$ &  $N=200$    & $N=400$     & $N=800$     &  $N=1600$    & $N=3200$     & Rate        \\ \hline
				\multirow{6}{*}{$(1.3,-1.8)$}
				& 2    & 1.0201e-05  & 9.9349e-07  & 2.1620e-07  & 4.7054e-08  & 1.0240e-08  & 2.2000     \\
				& 3    & 1.1752e-05  & 4.5059e-10  & 4.8207e-11  & 5.1946e-12  & 5.6237e-13  & 3.2074     \\
				& 4    & 6.2468e-06  & 1.1298e-11  & 4.8878e-13  & 2.4804e-14  & 1.3228e-15  & 4.2288     \\
				& 5    & 7.0941e-07  & 5.1965e-12  & 7.8855e-14  & 1.2157e-15  & 1.8910e-17  & 6.0064     \\
				& 6    & 1.4520e-06  & 6.6647e-12  & 1.0078e-13  & 1.5492e-15  & 2.4008e-17  & 6.0118     \\
				& 7    & 1.1401e-06  & 8.2014e-12  & 1.2374e-13  & 1.8999e-15  & 2.9426e-17  & 6.0126     \\ \hline
				\multirow{6}{*}{$(1.7,-1.2)$}
				& 2    & 7.5836e-08  & 1.0904e-08  & 1.5662e-09  & 2.2492e-10  & 3.2298e-11  & 2.7999     \\
				& 3    & 3.6636e-09  & 1.4753e-10  & 9.1845e-12  & 6.3756e-13  & 4.5227e-14  & 3.8173     \\
				& 4    & 4.0340e-09  & 6.6212e-11  & 1.0526e-12  & 1.6515e-14  & 2.5691e-16  & 6.0063     \\
				& 5    & 5.5435e-09  & 8.9675e-11  & 1.4177e-12  & 2.2247e-14  & 3.4822e-16  & 5.9974     \\
				& 6    & 7.1565e-09  & 1.1455e-10  & 1.8021e-12  & 2.8214e-14  & 4.4114e-16  & 5.9990     \\
				& 7    & 8.8845e-09  & 1.4105e-10  & 2.2105e-12  & 3.4546e-14  & 5.3967e-16  & 6.0003     \\ \hline
			\end{tabular}\label{table:22}
	\end{center}}
\end{table}

\begin{table}[htbp]
	{\footnotesize \begin{center}
			\caption{Case {\bf(b)} with product: convergent order of ID$m$-BDF$6$.}
			\begin{tabular}{|l|l| l l l l l l|} \hline
				$(\gamma,\mu)$ &$m$ &  $N=200$    & $N=400$     & $N=800$     &  $N=1600$    & $N=3200$     & Rate        \\ \hline
				\multirow{6}{*}{$(1.3,-1.8)$}
				& 2    & 7.0394e-03  & 2.1947e-03  & 9.5571e-04  & 4.1608e-04  & 1.8113e-04  & 1.1998     \\
				& 3    & 2.7013e-04  & 1.8781e-06  & 4.0928e-07  & 8.9133e-08  & 1.9405e-08  & 2.1995     \\
				& 4    & 4.3584e-04  & 2.3163e-09  & 2.5035e-10  & 2.7125e-11  & 2.9450e-12  & 3.2033     \\
				& 5    & 3.0476e-04  & 2.6931e-11  & 9.6262e-13  & 4.5555e-14  & 2.3799e-15  & 4.2586     \\
				& 6    & 9.6121e-05  & 1.5891e-11  & 2.4273e-13  & 3.7507e-15  & 5.8309e-17  & 6.0072     \\
				& 7    & 5.3468e-06  & 1.7160e-11  & 2.6162e-13  & 4.0374e-15  & 6.2691e-17  & 6.0090     \\ \hline
				\multirow{6}{*}{$(1.7,-1.2)$}
				&2    & 3.6915e-04  & 1.0602e-04  & 3.0447e-05  & 8.7440e-06  & 2.5110e-06  & 1.7999     \\
				&3    & 3.3436e-07  & 4.9357e-08  & 7.1436e-09  & 1.0292e-09  & 1.4802e-10  & 2.7976     \\
				&4    & 1.9015e-08  & 5.2473e-10  & 3.0702e-11  & 2.1358e-12  & 1.5242e-13  & 3.8085     \\
				&5    & 1.6872e-08  & 2.4119e-10  & 3.7724e-12  & 5.9373e-14  & 9.2660e-16  & 6.0017     \\
				&6    & 1.6808e-08  & 2.4345e-10  & 3.8372e-12  & 6.0772e-14  & 9.5765e-16  & 5.9877     \\
				&7    & 1.6852e-08  & 2.4724e-10  & 3.9177e-12  & 6.2163e-14  & 9.8029e-16  & 5.9867     \\ \hline
			\end{tabular}\label{table:33}
	\end{center}}
\end{table}

	The numerical results using the scheme \eqref{2.3} are presented in Table \ref{table:11}, which indicates that ID$m$-BDF$k$ with $k=6$ recovers the high-order convergence, which is in agreement with Theorem \ref{theorem3.9}.
	In Table \ref{table:11}, we also observe the Newton-Cotes rule that appear in \cite{ShiHighorder2023}, which have an optimal convergence rate, i.e., $\mathcal{O}\left(\tau^{\min\{m+1,k\}}\right)$ for odd $m$ and $\mathcal{O}\left(\tau^{\min\{m+2,k\}}\right)$ for even $m$. Tables \ref{table:22} and \ref{table:33} show that ID$m$-BDF$6$ scheme can restore higher-order convergence, which is consistent with Theorems \ref{theorema5.2} and \ref{theorem5.5}, respectively.
	
	\subsection{Subdiffusion}
	Let $T=1$ and $\Omega=(-1,1)$.
	Let us consider the following two examples in \eqref{subfee} or \eqref{rsubfee}:
	\begin{description}
		\item[(c)] $\upsilon(x)=\sin(x)\sqrt{1-x^2}$ and $g(x,t)= t^{\mu} e^x \left(1+\chi_{\left(0,1\right)}\left(x\right)\right)$.
		\item[(d)] $\upsilon(x)=\sin(x)\sqrt{1-x^2}$ and $g(x,t)= (1+t^{\mu}) \circ e^t e^x \left(1+\chi_{\left(0,1\right)}\left(x\right)\right)$.
	\end{description}
	
\begin{table}[htbp]
	{\footnotesize \begin{center}
			\caption{Case {\bf(c)}: convergent order of ID$m$-BDF$6$.}
			\begin{tabular}{|l|l| l l l l l l|} \hline
				$(\gamma,\mu)$ &$m$ &  $N=200$    & $N=400$     & $N=800$     &  $N=1600$    & $N=3200$     & Rate        \\ \hline
				\multirow{6}{*}{$(0.3,-1.8)$}
				& 2    & 5.0725e-04  & 2.2042e-04  & 9.5865e-05  & 4.1710e-05  & 1.8151e-05  & 1.2003     \\
				& 3    & 1.4190e-06  & 3.0915e-07  & 6.7319e-08  & 1.4655e-08  & 3.1899e-09  & 2.1998     \\
				& 4    & 1.4492e-09  & 1.4972e-10  & 1.6063e-11  & 1.7384e-12  & 1.8870e-13  & 3.2036     \\
				& 5    & 5.3894e-11  & 2.1579e-13  & 3.2583e-14  & 2.3768e-15  & 1.3870e-16  & 4.0989     \\
				& 6    & 8.3905e-11  & 1.2443e-12  & 1.8941e-14  & 2.9196e-16  & 4.5270e-18  & 6.0111     \\
				& 7    & 9.8088e-11  & 1.4550e-12  & 2.2157e-14  & 3.4180e-16  & 5.3065e-18  & 6.0092     \\ \hline
				\multirow{6}{*}{$(0.7,-1.2)$}
				& 2    & 3.3359e-05  & 9.5745e-06  & 2.7487e-06  & 7.8925e-07  & 2.2663e-07  & 1.8001     \\
				& 3    & 2.1419e-08  & 3.1187e-09  & 4.5080e-10  & 6.4944e-11  & 9.3404e-12  & 2.7976     \\
				& 4    & 1.9006e-10  & 1.4558e-11  & 1.0580e-12  & 7.6129e-14  & 5.4673e-15  & 3.7995     \\
				& 5    & 2.1496e-11  & 3.2580e-13  & 5.1577e-15  & 8.6210e-17  & 1.5742e-18  & 5.7751     \\
				& 6    & 2.5590e-11  & 3.8163e-13  & 5.8267e-15  & 9.0001e-17  & 1.3982e-18  & 6.0082     \\
				& 7    & 2.9909e-11  & 4.4596e-13  & 6.8080e-15  & 1.0515e-16  & 1.6335e-18  & 6.0083     \\ \hline
			\end{tabular}\label{table:44}
	\end{center}}
\end{table}
\begin{table}[htbp]
	{\footnotesize \begin{center}
			\caption{Case {\bf(c)} with convolution: convergent order of ID$m$-BDF$6$.}
			\begin{tabular}{|l|l| l l l l l l|} \hline
				$(\gamma,\mu)$ &$m$ &  $N=200$    & $N=400$     & $N=800$     &  $N=1600$    & $N=3200$     & Rate        \\ \hline
				\multirow{6}{*}{$(0.3,-1.8)$}
				& 2    & 4.5639e-07  & 9.9423e-08  & 2.1648e-08  & 4.7126e-09  & 1.0257e-09  & 2.1998     \\
				& 3    & 6.8806e-10  & 7.2685e-11  & 7.8306e-12  & 8.4841e-13  & 9.2131e-14  & 3.2030     \\
				& 4    & 1.3338e-11  & 1.9377e-13  & 2.3273e-14  & 1.4656e-15  & 8.2840e-17  & 4.1450     \\
				& 5    & 3.8346e-11  & 5.7240e-13  & 8.7397e-15  & 1.3489e-16  & 2.0921e-18  & 6.0107     \\
				& 6    & 5.8079e-11  & 8.6726e-13  & 1.3249e-14  & 2.0471e-16  & 3.1809e-18  & 6.0080     \\
				& 7    & 8.1817e-11  & 1.2217e-12  & 1.8664e-14  & 2.8839e-16  & 4.4809e-18  & 6.0080     \\ \hline
				\multirow{6}{*}{$(0.7,-1.2)$}
				& 2    & 6.6843e-09  & 9.7230e-10  & 1.4050e-10  & 2.0237e-11  & 2.9104e-12  & 2.7977     \\
				& 3    & 1.0016e-10  & 7.3535e-12  & 5.3001e-13  & 3.8070e-14  & 2.7330e-15  & 3.8000     \\
				& 4    & 6.6723e-12  & 1.0297e-13  & 1.6883e-15  & 3.0182e-17  & 6.1603e-19  & 5.6145     \\
				& 5    & 1.1038e-11  & 1.6503e-13  & 2.5226e-15  & 3.8988e-17  & 6.0589e-19  & 6.0078     \\
				& 6    & 1.6689e-11  & 2.4948e-13  & 3.8132e-15  & 5.8930e-17  & 9.1574e-19  & 6.0079     \\
				& 7    & 2.3486e-11  & 3.5106e-13  & 5.3656e-15  & 8.2920e-17  & 1.2885e-18  & 6.0079     \\ \hline
			\end{tabular}\label{table:55}
	\end{center}}
\end{table}
\begin{table}[htbp]
	{\footnotesize \begin{center}
			\caption{Case {\bf(c)} with product: convergent order of ID$m$-BDF$6$.}
			\begin{tabular}{|l|l| l l l l l l|} \hline
				$(\gamma,\mu)$ &$m$ &  $N=200$    & $N=400$     & $N=800$     &  $N=1600$    & $N=3200$     & Rate        \\ \hline
				\multirow{6}{*}{$(0.3,-1.8)$}
				& 2    & 5.0689e-04  & 2.2034e-04  & 9.5847e-05  & 4.1706e-05  & 1.8150e-05  & 1.2002     \\
				& 3    & 1.4193e-06  & 3.0918e-07  & 6.7322e-08  & 1.4655e-08  & 3.1899e-09  & 2.1998     \\
				& 4    & 1.4529e-09  & 1.4968e-10  & 1.6057e-11  & 1.7380e-12  & 1.8868e-13  & 3.2034     \\
				& 5    & 6.1820e-11  & 3.1557e-13  & 3.0827e-14  & 2.3493e-15  & 1.3827e-16  & 4.0866     \\
				& 6    & 9.3452e-11  & 1.3869e-12  & 2.1119e-14  & 3.2563e-16  & 5.0501e-18  & 6.0107     \\
				& 7    & 1.0924e-10  & 1.6217e-12  & 2.4703e-14  & 3.8114e-16  & 5.9179e-18  & 6.0091     \\ \hline
				\multirow{6}{*}{$(0.7,-1.2)$}
				& 2    & 3.3358e-05  & 9.5743e-06  & 2.7487e-06  & 7.8925e-07  & 2.2663e-07  & 1.8001     \\
				& 3    & 2.1431e-08  & 3.1196e-09  & 4.5087e-10  & 6.4948e-11  & 9.3408e-12  & 2.6985     \\
				& 4    & 1.8919e-10  & 1.4545e-11  & 1.0578e-12  & 7.6126e-14  & 5.4672e-15  & 3.7995     \\
				& 5    & 2.2575e-11  & 3.4198e-13  & 5.4055e-15  & 9.0047e-17  & 1.6339e-18  & 5.7842     \\
				& 6    & 2.6889e-11  & 4.0111e-13  & 6.1248e-15  & 9.4611e-17  & 1.4699e-18  & 6.0082     \\
				& 7    & 3.1428e-11  & 4.6873e-13  & 7.1566e-15  & 1.1054e-16  & 1.7173e-18  & 6.0082     \\ \hline
			\end{tabular}\label{table:66}
	\end{center}}
\end{table}

	Tables \ref{table:44}, \ref{table:55} and \ref{table:66} show that ID$m$-BDF6 recovers the high-order convergence for the fractional subdiffusion model with the hyper-singular terms, which is in agreement with Theorems \ref{theorem6.1}, \ref{theorem6.2} and \ref{theorem6.3}, respectively.

	%

	

\end{document}